\documentclass[12pt]{article}
\usepackage{mathrsfs}
\usepackage{amssymb}
\usepackage{color}

\usepackage{amsthm}
\usepackage{amsmath}
\usepackage{graphicx}

 \newtheorem{thm}{Theorem}[section]
 
 \newtheorem{lem}[thm]{Lemma}
 
 \theoremstyle{definition}
 
 \newtheorem{rem}[thm]{Remark}
 \numberwithin{equation}{section}

\theoremstyle{definition}

\theoremstyle{remark}

\begin{document}
\title{Global Well-posedness of Incompressible Elastodynamics in Two Dimensions}
\author{Zhen
 Lei\footnote{School of Mathematical Sciences; LMNS and Shanghai
 Key Laboratory for Contemporary Applied Mathematics, Fudan University, Shanghai 200433, P. R.China and
 Institute for Advanced Study, Princeton University, NJ 08540, USA. {\it Email:
 leizhn@gmail.com, leizhn@math.ias.edu}}}
\date{Fudan University and Institute for Advanced Study}
\maketitle

\begin{abstract}
We prove that for sufficiently small initial displacements in some
weighted Sobolev space, the Cauchy problem of the systems of incompressible isotropic Hookean elastodynamics in two space dimensions
admits a uniqueness global classical solution.
\end{abstract}

\maketitle





\section{Introduction}

This paper considers the global existence of  classical solutions to
the Cauchy problem in incompressible nonlinear elastodynamics. The
elastic body is assumed to be homogeneous, isotropic and
hyperelastic. The systems of equations describing the motion exhibit
a nonlocal nature when one solves the pressure by inverting a Laplacian. The linearized
system turns out to be of wave type. We exploit the fact that the
nonlinearities in the systems of incompressible isotropic Hookean\footnote{The general case will be treated in a forthcoming paper using similar ideas as those in this article.}
elastodynamics are \textit{inherently strong linearly degenerate}
and automatically satisfy a \textit{strong null condition}.
We prove the global
existence of classical solutions to this Cauchy problem in the
two-dimensional case for small initial displacements in a certain
weighted Sobolev space.

To place our result in context, we review a few highlights from the
existence theory of nonlinear wave equations and elastodynamics. If
the spatial dimension is no bigger than three, the global existence
of these equations hinges on two basic assumptions (see \cite{Sideris00}): the initial data
being sufficiently small and the nonlinearities satisfying a type of
null condition. The absence of either of these conditions may lead
to the finite time blowup of solutions. In particular, for 3D
compressible elastodynamics, John \cite{John84} proved the formation of
finite time singularities for
arbitrarily small spherically symmetric displacements without the null condition. On the other hand,
Tahvildar-Zadeh \cite{T98} proved the formation
of singularities for large displacements. For nonlinear wave equations with
sufficiently small initial data but without the null condition, the
finite time blowup was shown by John \cite{John81} and Alinhac
\cite{Alinhac00b} in 3D, and Alinhac \cite{Alinhac99a, Alinhac99b,
Alinhac01} in 2D. We remark that the 2D case is highly nontrivial. See also some related classical results by Sideris \cite{Sideris1,Sideris2}.

From now on, we will always assume that the initial data is
sufficiently small in certain weighted Sobolev space since we are
concerned with the long-time behavior of nonlinear elastodynamics
and nonlinear wave equations. The first non-trivial long time
existence result may be the one by John and Klainerman
\cite{JohnKlainerman84} where the almost global existence theory is
obtained for  3D quasilinear scalar wave equation. In the seminal
work \cite{Klainerman85}, Klainerman introduced the vector field
theory and the so-called generalized energy method based on the
scaling, rotation and Lorentz invariant properties of wave operator,
providing a general framework for studying nonlinear wave equations.
Then in \cite{Klainerman86}, Klainerman  proved the global existence
of classical solutions for 3D scalar quasilinear wave equations under
the assumptions that the nonlinearity satisfies the null condition.
This landmark work was also obtained independently by Christodoulou
\cite{Christodoulou86} using a conformal mapping method. We also
mention that John \cite{John88} established the almost global
existence theory for 3D compressible elastodynamics via an
$L^1-L^\infty$ estimate.

The generalized energy method of Klainerman can be adapted to prove
almost global existence for certain nonrelativistic systems of 3-D
nonlinear wave equations, using Klainerman-Sideris's weighted $L^2$
energy estimate which only involves the scaling and rotation
invariance of the system, as was first done in \cite{KS96}. This
approach was subsequently developed to obtain the global existence under
the null condition in \cite{SiderisTu01}, see also \cite{Yokoyama}
for a different method. Of particular importance is that Sideris
\cite{Sideris00} (independently by Agemi \cite{Agemi00}, see also an
earlier result \cite{Sideris96}) proved the global existence of
classical solutions to the 3D compressible elastodynamics under a
null condition. For 3D incompressible elastodynamics, the Hookean
part of the system is inherently degenerate and satisfies a null
condition. The global existence was established by Sideris and
Thomases in \cite{ST05, ST07}. We would like to point out that a unified
treatment for obtaining weighted $L^2$ estimates (of second order
derivatives of unknowns) for certain 3D hyperbolic systems appeared
in \cite{ST06}.

As having been pointed out in  \cite{LSZ13}, the existence question is more delicate in 2-D case because, even
under the assumption of the null condition, quadratic nonlinearities
have at most critical time decay. A series of articles considered
the case of cubically nonlinear equations satisfying the null
condition, see for example \cite{Hoshiga,Katayama,L95}. Alinhac
\cite{Alinhac00} was the first to establish the global existence for 2D
scalar wave equation with null bilinear forms. His argument combines
vector fields  with what he called the ghost energy method. We
emphasis that Alinhac took the advantage of the Lorentz
invariance of the system. In particular, at the time of this
writing, under null conditions, the global well-posedness of
following problems (for which Lorentz invariance is not available)
still remains widely open:
\begin{itemize}\nonumber
\item Nonlinear nonrelativistic wave
systems in two dimension.
\item Systems of nonlinear elastodynamics in two dimension.
\end{itemize}
The first non-trivial long time existence result concerning the
above two problems is the recent work by Lei, Sideris and Zhou
\cite{LSZ13}, in which the authors proved the almost global
existence for incompressible isotropic elastodynamics in 2D by
formulating the system in Eulerian coordinate.  Their proof is based on
Klainerman's generalized energy method, enhanced by Alinhac's ghost
energy method and Klainerman-Sideris's weighted $L^2$ energy method.
That is the first time to apply Alinhac's ghost weight method to the
case in which the Lorentz invariance is absent and the system is nonlocal. A
novel observation is the treatment of the term involving pressure
which is shown to enjoy a null structure. Unfortunately, at present, it seems
hard for us to improve the result in \cite{LSZ13} to be a global one
under the framework there.

In this paper, we prove the global existence of the incompressible
elasticity in two dimensions. To better illustrate our ideas and to
make the presentation as simple as it could be, we will only focus on
the typical Hookean case, and will treat the general case in another paper using similar ideas.  
Our proof here is a more structural one than
requiring involved technical tools. We first formulate the system in
Lagrangian coordinate. Let $p$ denote the pressure and $X$ the flow map. Let $Y(t, y) = X(t, y) - y$. The first observation is that the main part
of the main nonlinearity $\nabla\Gamma^\alpha p$ always contains a
term of $(\partial_t^2 - \Delta)\Gamma^\alpha Y$ or $(\partial_t + \partial_r)D\Gamma^\alpha Y$. Here $D$ denotes a space or time derivative and $\Gamma$ is a vector field which is defined in section 2, above \eqref{Elasticity-D}. This gives us the
so-called \textit{strong null condition}\footnote{When the usual \textit{null condition} is satisfied by a quasi-linear wave systems, the nonlinearities contain the term of  $(\partial_t + \partial_r)\Gamma^\alpha Y$ in general.} (we suggest to use this
terminology). When we perform the highest order
generalized energy estimate, at the first glance we always lose one
derivative if we bound the pressure term in $L^2$ (see Section
\ref{WE} and Section \ref{HEE}). A natural way to avoid this
difficulty may be introducing the new unknowns
$$U^\alpha = \Gamma^\alpha Y + (\nabla X)^{- \top}\nabla(-
\Delta)^{-1}\nabla\cdot[(\nabla X)^\top\Gamma^\alpha Y]$$ to
symmetrize the system. Unfortunately, this idea leads to complicated
calculations, and more essentially, it may not work at all. But
fortunately, the inherent strong null structure of nonlinearities
helps us to obtain a kind of estimate in which we gain one derivative
when estimating the $L^2$ norm of $\nabla\Gamma^\alpha p$ (see
Section \ref{HEE}). The price we pay here is to have less decay rate
in time, which will be overcome by applying Alinhac's ghost weight
as in \cite{Agemi00}. For the lower order energy estimate, our
observation is that instead of estimating the $L^2$ norm of
$D\Gamma^\alpha Y$, we turn to estimate its divergence-free part and
curl-free part. The divergence-free part is not a main problem. When
estimating the curl-free part of $\Gamma^\alpha Y$, the strong
linearly degenerate structure is present once again by appropriately
rewriting the system as the form in \eqref{Elasticity-L}. Then the
generalized energy estimate for the curl-free part of $\Gamma^\alpha
Y$ can be carried out with a subcritical time decay
which is $\langle t\rangle^{- \frac{3}{2}}$ (it is still not clear
for us whether the similar estimate is true in Eulerian coordinate).

As in \cite{KS96, Sideris00}, we need to estimate a kind of weighted
$L^2$ generalized energy norm. A technical difficulty here is to
control the weighted generalized $L^2$ energy $\mathcal{X}_\kappa$
in terms of the generalized energy $\mathcal{E}_\kappa$. For this
purpose, we will have to estimate the $r$-weighted null form of
$r(\partial_t^2 - \Delta)Y$, which seems not naively true since $r$
is not in $\mathcal{A}_p$ class of a zero-order Riesz operator for
$p = 2$ in two space dimension (See Lemma \ref{WE-1} for details).
We overcome this difficulty by considering a variety of the weighted
$L^2$ generalized energy of Klainerman-Sideris \cite{KS96} (see its
definition \eqref{Defn} in Section \ref{Equations}) so that we only
need to estimate $\langle t\rangle(\partial_t^2 - \Delta)Y$. A trick
here is that instead of a gain of one derivative as in performing
highest order energy estimate in section \ref{HEE}, we use the
strong null condition to gain suitable decay rate in time.

Before ending the introduction, let us mention some related works on
viscoelasticity, where there is viscosity in the momentum equation.
The global well-posedness with near equilibrium states in 2-D is first
obtained in \cite{LLZ2005}. The 3-D case
was obtained independently in  \cite{CZ2006} (see also the thesis
\cite{LeiS}) and \cite{LeiLZ08}. The initial boundary value problem
is considered in \cite{LL2008}, the compressible case can be found
in \cite{QZ2010, HW2011}. For more results near equilibrium, readers
are referred to the nice review paper by Lin \cite{Lin} and other
works in \cite{LZ2005, HX2010,ZF-1,ZF-2} as references. In \cite{Lei} a
class of large solutions in two space dimensions are established via
the strain-rotation decomposition (which is based on earlier results
in \cite{Lei2008} and \cite{Lei3}). In all of these works, the
initial data is restricted by the viscosity parameter. The work
\cite{Kessenich} was the first to establish global existence for 3-D
viscoelastic materials uniformly in the viscosity parameter. We also mention that Hao and Wang  recently established the local a priori estimate for the free boundary incompressible elastodynamics in \cite{HW}.

We will give a self-contained presentation for the whole proof. The
remaining part of this paper is organized as follows: In section
\ref{Equations} we will formulate the system of incompressible
elastodynamics in Lagrangian coordinate and present its basic
properties, then we introduce some notations and state the main
result of this paper. We will outline the main steps of the proof at
the end of this section. In section \ref{lemmas} we will prove some
weighted Sobolev imbedding inequalities, weighted $L^\infty$
estimate and a refined Sobolev inequality. We also
give the estimate for good derivatives and lay down a preliminary
step for estimating weighted generalized $L^2$ energy. Then we will
explore the estimate for strong null form in section \ref{WE}, and
at the end of that section we give the estimate for the weighted
$L^2$ energy. In section \ref{HEE} we present the highest order
generalized energy estimate. Then we perform the lower order
generalized energy estimate in section \ref{LEE}.

\begin{rem}
After posting this article on Arxiv (see  arXiv:1402.6605), the author was informed by Dr. Wang that he could give another proof of the main result which, as being claimed in Wang's paper, can improve the understanding of the behavior of solutions in different coordinates using a different approach and from the point of view in frequency space (see Xuecheng Wang, \textit{Global existence for the 2D incompressible isotropic elastodynamics for small initial data}, arXiv:1407.0453).
\end{rem}

\section{Equation and Its Basic Properties}\label{Equations}

In the incompressible case, the equations of elastodynamics are in
general more conveniently written as a first-order system with
constraints in Eulerian coordinates (see, for instance \cite{ST05,
ST07, LSZ13}). But we will formulate the system in Lagrangian
coordinate below.

For any given smooth flow map $X(t, y)$, we call it incompressible
if
\begin{equation}\nonumber
\int_{\Omega}dy = \int_{\Omega_t}dX,\quad \Omega_t = \{X(t, y)|y \in
\Omega\}
\end{equation}
for any smooth bounded connected domain $\Omega$. Clearly, the
incompressibility is equivalent to
\begin{equation}\nonumber
\det(\nabla X) \equiv 1.
\end{equation}
Denote
\begin{equation}\label{Perturb-fm}
X(t, y) = y + Y(t, y).
\end{equation}
Then we have
\begin{equation}\label{Struc-1}
\nabla\cdot Y = - \det(\nabla Y).
\end{equation}
Moreover, a simple calculation shows that
\begin{equation}\label{Inverse}
(\nabla X)^{- 1} = \begin{pmatrix} 1 + \partial_2Y^2 & -
\partial_1Y^2\\ - \partial_2Y^1 & 1 + \partial_1Y^1\end{pmatrix} = (\nabla\cdot X)I - (\nabla X)^\top.
\end{equation}
We remark that throughout this paper we use the following convention
$$(\nabla Y)_{ij} = \frac{\partial Y^i}{\partial y^j}.$$
We often use the following notations
\begin{equation}\nonumber
\omega = \frac{y}{r},\quad r = |y|,\quad \omega^\perp =
\begin{pmatrix}- \omega_2, \omega_1\end{pmatrix},\quad \nabla^\perp =
\begin{pmatrix}- \partial_2, \partial_1\end{pmatrix}.
\end{equation}

For homogeneous, isotropic and hyperelastic materials, the motion of
the elastic fluids is determined by the following Lagrangian
functional of flow maps:
\begin{eqnarray}\label{LaF}
\mathcal{L}(X; T, \Omega) &=&
\int_0^T\int_{\Omega}\big(\frac{1}{2}|\partial_tX(t, y)|^2 -
W(\nabla X(t, y))\\\nonumber && +\ p(t, y)\big[\det\big(\nabla X(t,
y)\big) - 1\big]\big)dydt.
\end{eqnarray}
Here $W \in C^\infty(GL_2, \mathbb{R}_+)$ is the strain energy
function which depends only on $F = \nabla X$ and $p(t, y)$ is a Lagrangian
multiplier which is used to force the flow maps to be
incompressible. We call that $X(t, y)$ is a critical point of
$\mathcal{L}$ if for a given $T \in (0, \infty)$ and  bounded smooth
connected domain $\Omega$, there holds
\begin{equation}\nonumber
\frac{d}{d\epsilon}\Big|_{\epsilon = 0}\mathcal{L}(X^\epsilon; T, \Omega) =
0
\end{equation}
for all $p \in C^1(\mathbb{R}_+\times \mathbb{R}^2 \rightarrow
\mathbb{R})$ and any one-parameter family of incompressible flow
maps $X^\epsilon \in C^1(\mathbb{R}_+ \times \mathbb{R}^2
\rightarrow \mathbb{R}^2)$ with $\frac{d}{dt}\Big|_{\epsilon =
0}X^\epsilon(t, y) = Z(t, y)$, $X^0(t, y) = X(t, y)$ and  $Z(0, y) =
Z(T, y) \equiv 0$ for all $y \in \Omega$,  $Z(t, y) \equiv 0$ for
all $t \in [0, T]$ and $y \in
\partial\Omega$.

Let us focus on the most simplest case, i.e., the Hookean case, in
which the strain energy functional is simply given by
$$W(\nabla X) = \frac{1}{2}|\nabla X|^2.$$
Clearly, the Euler-Lagrangian equation of \eqref{LaF} takes
\begin{equation}\label{B-9}
\begin{cases}
\partial_t^2Y - \Delta Y = - (\nabla X)^{- \top}\nabla p,\\[-4mm]\\
\nabla\cdot Y = - \det(\nabla Y).
\end{cases}
\end{equation}

Let us take a look at the invariance groups of system \eqref{B-9}.
Suppose that $X(t, y)$ is a critical point of $\mathcal{L}$ in
\eqref{LaF}. Clearly, $\widetilde{X}(t, y)$ are also critical points
of $\mathcal{L}$ in \eqref{LaF}, which are defined either by
\begin{equation}\label{RoT}
\widetilde{X}(t, y) = Q^TX(t, Q y),\quad \forall\quad Q = e^{\lambda
A},\ \ A = \begin{pmatrix}0 & - 1\\ 1 & 0\end{pmatrix},
\end{equation}
or by
\begin{equation}\label{SaT}
\widetilde{X}(t, y) = \lambda^{-1}X(\lambda t, \lambda y)
\end{equation}
for all $\lambda > 0$. As a result, one also has
\begin{equation}\label{2-1}
\begin{cases}
\partial_t^2\widetilde{Y} - \Delta \widetilde{Y} = - (\nabla \widetilde{X})^{- \top}\nabla \widetilde{p},\\[-4mm]\\
\nabla\cdot \widetilde{Y} = - \det(\nabla \widetilde{Y}),
\end{cases}
\end{equation}
where $\widetilde{p}(t, y) = p(t, Qy)$ if $\widetilde{X}$ is defined
by \eqref{RoT} and $\widetilde{p}(t, y) = p(\lambda t, \lambda y)$
if $\widetilde{X}$ is defined by \eqref{SaT}.

Let us first take $\widetilde{X}$ being defined in \eqref{SaT}.
Differentiate \eqref{2-1} with respect to $\lambda$ and then take
$\lambda = 1$, one has
\begin{equation}\label{2-2}
\begin{cases}
(\nabla X)^{\top}(\partial_t^2 - \Delta)(S - 1)Y + [\nabla
  (S - 1)Y]^{\top}(\partial_t^2 - \Delta) Y = - \nabla Sp,\\[-4mm]\\
\nabla\cdot (S - 1)Y = \partial_1(S - 1)Y^2\partial_2Y^1 +
  \partial_1Y^2\partial_2(S - 1)Y^1\\
\quad\quad\quad\ -\ \partial_1Y^1\partial_2(S - 1)Y^2 - \partial_1(S
  - 1)Y^1\partial_2Y^2.
\end{cases}
\end{equation}
Here $S$ denotes the scaling operator: $$S = t\partial_t +
y^j\partial_j,$$ and throughout this paper, we use Einstein's
convention for repeated indices. Similarly, denote the rotation
operator by
\begin{equation}\nonumber
\Omega = I\partial_\theta + A,\quad \partial_\theta = y^1\partial_2
- y^2\partial_1,
\end{equation}
where $A$ is given in \eqref{RoT}. Let $\widetilde{X}$ be defined in
\eqref{RoT}. Differentiate \eqref{2-1} with respect to $\lambda$ and
then take $\lambda \to 0$, one has
\begin{equation}\label{2-3}
\begin{cases}
(\nabla X)^{\top}(\partial_t^2\Omega Y - \Delta \Omega Y) + (\nabla
  \Omega Y)^{\top}(\partial_t^2Y - \Delta Y) = - \nabla \partial_\theta p,\\[-4mm]\\
\nabla\cdot \Omega Y = \partial_1\Omega Y^2\partial_2Y^1 +
\partial_1Y^2\partial_2\Omega Y^1 -
\partial_1Y^1\partial_2\Omega Y^2 - \partial_1\Omega Y^1\partial_2Y^2.
\end{cases}
\end{equation}

For any  vector $Y$ and scalar $p$, we make the following
conventions:
\begin{equation}\nonumber
\begin{cases}
\widetilde{\Omega} Y \triangleq \partial_\theta Y + AY, \quad
\widetilde{\Omega} p
\triangleq \partial_\theta p,\quad \widetilde{\Omega} Y^j = (\widetilde{\Omega} Y)^j,\\[-4mm]\\
\widetilde{S} Y \triangleq (S - 1) Y, \quad \widetilde{S} p
\triangleq S p,\quad \widetilde{S} Y^j = (\widetilde{S} Y)^j.
\end{cases}
\end{equation}
Let $\Gamma$ be any operator of the following set
\begin{equation}\nonumber
\{\partial_t, \partial_1, \partial_2, \widetilde{\Omega},
\widetilde{S}\}.
\end{equation}
Then for any multi-index $\alpha = \{\alpha_1, \alpha_2, \alpha_3,
\alpha_4, \alpha_5\}^\top \in \mathbb{N}^5$, using similar arguments as in \eqref{2-2} and
\eqref{2-3}, we have (see Appendix A)
\begin{eqnarray}\label{Elasticity-D}
&&\partial_t^2\Gamma^\alpha Y - \Delta \Gamma^\alpha Y = - (\nabla
  X)^{- \top}\nabla\Gamma^\alpha p\\\nonumber
&&\quad\quad\quad\quad\quad -\ \sum_{\beta + \gamma
  = \alpha,\ \gamma \neq \alpha}C_\alpha^\beta(\nabla X)^{- \top}(\nabla
  \Gamma^\beta Y)^{\top}(\partial_t^2 - \Delta)\Gamma^\gamma Y.
\end{eqnarray}
together with the constraint
\begin{equation}\label{Struc-2}
\nabla\cdot \Gamma^\alpha Y = \sum_{\beta + \gamma =
\alpha}C_\alpha^\beta\big(\partial_1\Gamma^\beta Y^2\partial_2\Gamma^\gamma Y^1 -
  \partial_1\Gamma^\gamma Y^1\partial_2\Gamma^\beta Y^2\big).
\end{equation}
Here the binomial coefficient $C_\alpha^\beta$ is given by
$$C_\alpha^\beta = \frac{\alpha!}{\beta!(\alpha - \beta)!}.$$ The structural identity \eqref{Struc-2} will be of extremal
importance for our proof.

Throughout this paper, we use the notation $D$ for space-time
derivatives:
$$D = (\partial_t, \partial_1,
\partial_2).$$
We use $\langle a\rangle$ to denote
$$\langle a\rangle = \sqrt{1 + a^2}$$
and $[a]$ to denote the biggest integer which is no more than $a$:
$$[a] = {\rm the\ biggest\ integer\ which\ is\ no\ more\ than\ a}.$$
We often use the following abbreviations:
$$\|\Gamma^{\leq |\alpha|}f\| = \sum_{|\beta| \leq |\alpha|}\|\Gamma^\beta f\|.$$
We need to use Klainerman's generalized energy which is defined by
\begin{equation}\nonumber
\mathcal{E}_{\kappa} = \sum_{|\alpha| \leq \kappa -
1}\|D\Gamma^\alpha Y\|_{L^2}^2.
\end{equation}
We define the following weighted $L^2$ generalized energy by (which
is a modification of the original one of Klainerman-Sideris in
\cite{KS96})
\begin{equation}\label{Defn}
\mathcal{X}_{\kappa} = \sum_{|\alpha| \leq \kappa - 2}\Big(\int_{r
\leq 2\langle t \rangle}\langle t - r\rangle^2 |D^2\Gamma^\alpha
Y|^2dy + \int_{r > 2\langle t \rangle}\langle t \rangle^2
|D^2\Gamma^\alpha Y|^2dy\Big).
\end{equation}
To describe the space of the initial data, we follow Sideris \cite{Sideris00} and introduce
\begin{equation}\nonumber
\Lambda = \{\nabla, r\partial_r, \Omega\},
\end{equation}
and
\begin{equation}\nonumber
H^\kappa_\Lambda =\big\{(f, g)\big| \sum_{|\alpha| \leq \kappa -
1}\big(\|\Lambda^\alpha f\|_{L^2} + \|\nabla\Lambda^\alpha f\|_{L^2}
+ \|\Lambda^\alpha g\|_{L^2}\big)<\infty\big\}.
\end{equation}
Then as in \cite{Sideris00}, we define
\begin{eqnarray}\nonumber
&&H^\kappa_\Gamma(T) = \big\{Y: [0,T)\times \mathbb{R}^2 \to
  \mathbb{R}^2\big|\Gamma^\alpha Y \in L^\infty([0,T);
  L^2(\mathbb{R}^2)),\\\nonumber
&&\quad \quad \quad\ \ \ \ \ \partial_t\Gamma^\alpha Y,
\nabla\Gamma^\alpha
  Y \in L^\infty([0,T); L^2(\mathbb{R}^2)),\ \ \forall\ |\alpha| \leq \kappa - 1\big\}
\end{eqnarray}
with the norm
\begin{equation}\nonumber
\sup_{0 \leq t < T} \mathcal{E}_\kappa^{1/2}(Y).
\end{equation}

We are ready to state the main theorem of this paper.

\begin{thm}\label{GlobalW}
Let $W(F) = \frac{1}{2}|F|^2$ be an isotropic Hookean strain energy
function. Let $M_0 > 0$ and $0 < \delta < \frac{1}{8}$ be two given
constants and $(Y_0, v_0) \in H^{\kappa}_\Lambda$ with $\kappa \ge
10$. Suppose that $Y_0$ satisfies the structural constraint
\eqref{Struc-1} at $t = 0$ and
\begin{eqnarray}\nonumber
&&\mathcal{E}_{\kappa}^{\frac{1}{2}}(0) = \sum_{|\alpha| \leq \kappa
- 1}\big(
  \|\nabla\Lambda^{\alpha}Y_0\|_{L^2} +
  \|\Lambda^{\alpha}v_0\|_{L^2}\big) \leq
  M_0,\\\nonumber
&&\mathcal{E}_{\kappa - 2}^{\frac{1}{2}}(0) = \sum_{|\alpha| \leq
\kappa - 3}
  \big(\|\nabla\Lambda^{\alpha}Y_0\|_{L^2} +
  \|\Lambda^{\alpha}v_0\|_{L^2}\big) \leq
  \epsilon.
\end{eqnarray}
There exists a positive constant $\epsilon_0 < e^{- M_0}$ which
depends only on $\kappa$ and $M_0$, $\delta$ such that, if $\epsilon
\leq \epsilon_0$, then the system of incompressible Hookean
elastodynamics \eqref{B-9} with following initial data
$$Y(0, y) = Y_0(y),\quad \partial_tY(0, y) = v_0(y)$$ has a unique
global classical solution such that
\begin{eqnarray}\nonumber
\mathcal{E}_{\kappa - 2}^{\frac{1}{2}}(t) \leq \epsilon
\exp\big\{C_0^2M_0\big\},\quad \mathcal{E}_{\kappa}^{\frac{1}{2}}(t)
\leq C_0M_0(1 + t)^{\delta}
\end{eqnarray}
for some $C_0 > 1$ uniformly in $t$.
\end{thm}

The main strategy of the proof is as follows: For initial data
satisfying the constraints in Theorem \ref{GlobalW}, we will prove
that
\begin{eqnarray}\label{C1}
\mathcal{E}_{\kappa}^\prime(t) \leq \frac{C_0}{4}\langle
t\rangle^{-1}\mathcal{E}_{\kappa}(t)\mathcal{E}_{\kappa -
  2}^{\frac{1}{2}}(t)
\end{eqnarray}
which is given in \eqref{ES-2} and
\begin{eqnarray}\label{C2}
\mathcal{E}_{\kappa - 2}^\prime(t) \leq \frac{C_0}{4}\langle
t\rangle^{-\frac{3}{2}}\mathcal{E}_{\kappa -
2}(t)\mathcal{E}_{\kappa}^{\frac{1}{2}}(t)
\end{eqnarray}
which is given in \eqref{LOEE-1} and \eqref{B-3}, for all $t \geq 0$ and some absolute
positive constant $C_0$ depending only on $\kappa$.  Once the above
differential inequalities are proved, it is easy to show that the
bounds for $\mathcal{E}_{\kappa - 2}^{\frac{1}{2}}$ and
$\mathcal{E}_{\kappa}^{\frac{1}{2}}$ given in the theorem hold true for all $t \geq 0$ by taking an appropriate small $\epsilon_0$, which yields the global existence result and
completes the proof of Theorem \ref{GlobalW}.

Indeed, one may just take
$$\epsilon_0 =  C_0^{-1}\delta\exp\big\{- C_0^2M_0\big\}.$$
Note that
$$\mathcal{E}_{\kappa - 2}^{\frac{1}{2}}(0) \leq \epsilon \leq \epsilon_0,\quad \mathcal{E}_{\kappa}^{\frac{1}{2}}(0) \leq M_0.$$
By continuity, there exists a positive time $T < \infty$  such that the bounds of $\mathcal{E}_{\kappa - 2}^{\frac{1}{2}}(t)$ and $\mathcal{E}_{\kappa}^{\frac{1}{2}}(t)$ in Theorem \ref{GlobalW} are true for $t \in [0, T]$:
\begin{eqnarray}\label{B-1}
\mathcal{E}_{\kappa - 2}^{\frac{1}{2}}(t) \leq \epsilon
\exp\big\{C_0^2M_0\big\} \leq C_0^{-1}\delta,\quad \mathcal{E}_{\kappa}^{\frac{1}{2}}(t)
\leq C_0M_0(1 + t)^{\delta}.
\end{eqnarray}

We claim that \eqref{B-1} is true for all $ t \in [0, \infty)$. We prove this claim by contradiction. Suppose that $T_{\rm max}\in [T, \infty)$ be the largest time such that \eqref{B-1} is true on $[0, T_{\rm max}]$. We are going to deduce a consequence which contradicts to the assumption on $T_{\rm max} < \infty$.
Keep in mind that now we have both \eqref{B-1} and the differential inequalities \eqref{C1}-\eqref{C2} in hand for $t \in [0,T_{\rm max}]$. By using \eqref{B-1} and the first differential inequality \eqref{C1}, one
has
\begin{eqnarray}\nonumber
\mathcal{E}_{\kappa}(t) \leq
\mathcal{E}_{\kappa}(0)\exp\Big\{\frac{\delta}{4}\int_0^t\langle t\rangle^{-1}dt\Big\} = M_0^2(1
+ t)^{\frac{\delta}{2}},\quad 0 \leq t \leq T_{\rm max}.
\end{eqnarray}
Similarly, by using \eqref{B-1} and the second differential inequality \eqref{C2}, one
has
\begin{eqnarray}\nonumber
\mathcal{E}_{\kappa - 2}(t) &\leq& \mathcal{E}_{\kappa - 2}
  (0)\exp\Big\{\frac{C_0^2M_0}{4}\int_0^t\langle
  s\rangle^{\delta - \frac{3}{2}}ds\Big\}\\\nonumber
&<& \epsilon^2\exp\big\{C_0^2M_0\big\},\quad 0 \leq t \leq T_{\rm max}.
\end{eqnarray}
Consequently, we have proved that, by taking
$$\epsilon_0 =  C_0^{-1}\delta\exp\big\{- C_0^2M_0\big\},$$
one has
\begin{eqnarray}\nonumber
\mathcal{E}_{\kappa - 2}(t) < \epsilon^2\exp\big\{C_0^2M_0\big\},\quad \mathcal{E}_{\kappa}(t) < M_0^2(1
+ t)^{2\delta},\quad 0 \leq t \leq T_{\rm max}.
\end{eqnarray}
The above inequalities show that \eqref{B-1} can still be true for $t \in [0, T_{\rm max} + \epsilon']$ for some $\epsilon' > 0$.  This contracts to the assumption on $T_{\rm max}$.
Hence we in fact
proved the \textit{a priori} bounds \eqref{B-1} on $[0, \infty)$ which is stated in Theorem
\ref{GlobalW}. Moreover, we have
$$\mathcal{E}_{\kappa - 2}^{\frac{1}{2}} \leq \epsilon
\exp\big\{C_0^2M_0\big\} \leq C_0^{-1}\delta.$$

So from now on our main goal is going to show the two \textit{a
priori} differential inequalities \eqref{C1}-\eqref{C2}. The highest order
one will be done in section \ref{HEE} and the lower order one will
be done in section \ref{LEE}. By taking an appropriately large $C_0$ and an appropriately small $\delta$, we can assume that
$\mathcal{E}_{\kappa - 2}^{\frac{1}{2}} \ll 1$, which is always assumed in the remaining of
this paper. We often use the fact that $\|\nabla X\|_{L^\infty} \leq
3$ since $\|\nabla X - I\|_{L^\infty} \lesssim \mathcal{E}_{\kappa -
2}^{\frac{1}{2}}$. Similarly, by \eqref{Inverse}, the above is also
applied for $(\nabla X)^{- 1}$.

\section{Preliminaries}\label{lemmas}

In this section we derive several weighted $L^\infty-L^2$ type decay
in time estimates. We remark that the idea of part of the proofs are
basically appeared in the earlier work \cite{LSZ13} and the references therein. The new
weighted $L^2$ energy $\mathcal{X}$ makes the proofs slightly
different. For a self-contained presentation, we still include their
proofs below.

We shall need apply the Littlewood-Paley theory. Let $\phi$ be a
smooth function supported in $\{\tau \in \mathbb{R}^+: \frac{3}{4}
\leq \tau \leq \frac{8}{3}\}$ such that
\begin{equation}\nonumber
\sum_{j\in \mathbb{Z}}\phi(2^{-j}\tau) = 1.
\end{equation}
For $f \in L(\mathbb{R}^2)$, we set
\begin{equation}\nonumber
\Delta_jf = \mathcal{F}^{-1}\big(\phi(2^{-
j}|\xi|)\mathcal{F}(f)\big)
\end{equation}
and
\begin{equation}\nonumber
S_{j}f = \sum_{- \infty < k \leq j - 1,\ k \in \mathbb{Z}}\Delta_kf.
\end{equation}
Here $\mathcal{F}$ denotes the usual Fourier transformation in $y$-variable and
$\mathcal{F}^{- 1}$ the inverse Fourier transformation.

The following lemma takes care of the decay properties of $L^\infty$
norm of derivative of unknowns. The 3D version of some of those kind of estimates has already
appeared in the work of Klainerman \cite{Klainerman85}, Klainerman-Sideris
\cite{KS96}, Sideris \cite{Sideris00} and the references therein. See also \cite{LSZ13} for some 2D cases. It shows that the $L^\infty$ norm of
derivative of unknowns will decay in time at least as $\langle
t\rangle^{- \frac{1}{2}}$. This can be improved a little bit to get
an extra factor $\langle t - r\rangle^{- \frac{1}{2}}$ near the
light cone region $\langle t\rangle/2 \leq r \leq 3\langle
t\rangle/2$. By a refined Sobolev imbedding inequality, one can even
improve the decay rate in time to be $\langle
t\rangle^{-1}\ln^{\frac{1}{2}}(e + t)$ in the space-time region away
from the light cone (we remark that in this paper we don't need the
full strength of the estimate in \eqref{C-11}). This will be used to
break the criticality of lower-order generalized energy estimate in
the space-time region $|r - t| \geq \frac{\langle t\rangle}{2}$. It
also exhibits that the lack of Lorentz invariance only leads to a
loss of time decay of $\ln^{\frac{1}{2}}(e + t)$ in \eqref{C-11}.

\begin{lem}\label{DecayEF}
Let $t \geq 4$. Then there holds
\begin{equation}\label{C-1}
\langle r \rangle^{\frac{1}{2}}|D\Gamma^\alpha Y| \lesssim \|D\Gamma^{\leq 2}\Gamma^\alpha Y\|_{L^2} \leq
\mathcal{E}_{|\alpha| + 3}^{\frac{1}{2}}.
\end{equation}
Moreover, for $r \leq 2\langle t\rangle /3$, or $r \geq 5\langle
t\rangle/4$, there holds
\begin{equation}\label{C-11}
 t |D\Gamma^\alpha Y| \lesssim
\big(\mathcal{E}_{|\alpha| + 1}^{\frac{1}{2}} +
\mathcal{X}_{|\alpha| + 3}^{\frac{1}{2}}\big)\ln^{\frac{1}{2}}\big(e
+ t\big).
\end{equation}
For $\langle t\rangle/3 \leq r \leq 5\langle t\rangle/2$, there
holds
\begin{equation}\label{C-3}
\langle r \rangle^{\frac{1}{2}} \langle t-r\rangle^{\frac{1}{2}}
|D\Gamma^\alpha Y| \lesssim \mathcal{E}_{|\alpha| + 2}^{\frac{1}{2}}
+ \mathcal{X}_{|\alpha| + 3}^{\frac{1}{2}}.
\end{equation}
\end{lem}
\begin{proof}
First of all, by Sobolev imbedding $H^2(\mathbb{R}^2)
\hookrightarrow L^\infty(\mathbb{R}^2)$, \eqref{C-1} is
automatically true for $r \leq 1$. By Sobolev imbedding on sphere
$H^1(\mathbb{S}^1) \hookrightarrow L^\infty(\mathbb{S}^1)$, one has
\begin{equation}\nonumber
|f(r\omega)|^2 \lesssim \sum_{|\beta| \leq
1}\int_{\mathbb{S}^1}|\Omega^\beta f(r\omega)|^2d\sigma.
\end{equation}
Consequently, 
we have
\begin{align*}
r|f(r\omega)|^2&\lesssim  \sum_{|\beta| \leq
  1} \int_{\mathbb{S}^1}r|\Omega^\beta f(r\omega)|^2d\sigma\\
&= - \sum_{|\beta| \leq 1}\int_{\mathbb{S}^1}d\sigma
  \int_r^\infty r\partial_\rho
  [|\Omega^\beta f(\rho\omega)|^2]d\rho\\
&= - \sum_{|\beta| \leq 1}\int_{\mathbb{S}^1}d\sigma
  \int_r^\infty r 2\Omega^\beta f(\rho \omega)\partial_\rho\Omega^\beta f(\rho \omega)d\rho\\
&\lesssim \sum_{|\beta| \leq 1}\int_{\mathbb{S}^1}\int_r^\infty
  |\Omega^\beta f(\rho\omega)||\partial_\rho\Omega^\beta f(\rho \omega)|\rho d\rho d\sigma\\
&\lesssim \sum_{|\beta| \leq 1}\|\partial_r\Omega^\beta
  f\|_{L^2}\|\Omega^\beta f\|_{L^2}.
\end{align*}
Then \eqref{C-1} follows by taking $f = D\Gamma^\alpha Y$ in the
above estimate.

Next, by the following well-known Bernstein inequality:
\begin{eqnarray}\nonumber
\|\Delta_jf\|_{L^\infty} \lesssim 2^{j}\|\Delta_jf\|_{L^2},\quad
\|S_jf\|_{L^\infty} \lesssim 2^{j}\|S_jf\|_{L^2},
\end{eqnarray}
one has
\begin{eqnarray}\nonumber
\|f\|_{L^\infty} &=& \|\sum_j\Delta_jf\|_{L^\infty}\\\nonumber &\lesssim&
  2^{- N}\|S_{- N}f\|_{L^2} + \sum_{- N \leq j \leq
  N }2^j\|\Delta_jf\|_{L^2} + \sum_{j \geq N +
  1}2^j\|\Delta_jf\|_{L^2}\\\nonumber
&\lesssim& 2^{- N}\|f\|_{L^2} + \sqrt{2N}\|\nabla f\|_{L^2} + 2^{-
  N}\|\nabla^2f\|_{L^2}.
\end{eqnarray}
Choosing $N = \frac{\ln(e + t)}{\ln 2}$, one has
\begin{equation}\label{sobolev}
\|f\|_{L^\infty} \lesssim \|\nabla f\|_{L^2}\ln^{\frac{1}{2}}\big(e
+ t\big) + \frac{1}{1 + t}\big(\|f\|_{L^2} +
\|\nabla^2f\|_{L^2}\big).
\end{equation}

Let us first choose a radial cutoff function $\varphi \in
C_0^\infty(\mathbb{R}^2)$ which satisfies
\begin{equation}\nonumber
\varphi = \begin{cases}1,\quad {\rm if}\ \frac{3}{4} \leq r \leq \frac{6}{5}\\
0,\quad {\rm if}\ r < \frac{2}{3}\ {\rm or}\ r >
\frac{5}{4}\end{cases},\quad |\nabla\varphi| \leq 100.
\end{equation}
For each fixed $t \geq 4$, let $\varphi^t(y) = \varphi(y/\langle
t\rangle)$. Clearly, one has
$$\varphi^t(y) \equiv 1\ \ {\rm for}\ \frac{3\langle t \rangle}{4} \leq r \leq
\frac{6\langle t \rangle}{5},\quad \varphi^t(y) \equiv 0\ \ {\rm
for}\ r \leq \frac{2\langle t \rangle}{3}\ {\rm or}\ r \geq
\frac{5\langle t \rangle}{4}$$ and $$|\nabla\varphi^t(y)| \leq
100\langle t\rangle^{-1}.$$

Consequently, for $r \leq 2\langle t\rangle /3$, or $r \geq 5\langle
t\rangle/4$, by  using \eqref{sobolev}, one has
\begin{eqnarray}\nonumber
 t |f| &\lesssim& \langle t\rangle \|(1
  - \varphi^t) f\|_{L^\infty}\\\nonumber
&\lesssim& \langle t\rangle\|\nabla[(1 - \varphi^t)
  f]\|_{L^2}\ln^{\frac{1}{2}}
  \big(e + t\big) + \|(1 - \varphi^t)f\|_{L^2}
  + \|\nabla^2[(1 - \varphi^t)f]\|_{L^2}\\\nonumber
&\lesssim& \big(\|f\|_{L^2} + \langle t\rangle\|(1 -
  \varphi^t)\nabla f\|_{L^2}\big)\ln^{\frac{1}{2}}\big(e + t\big)\\\nonumber
&&+\  \|f\|_{L^2} + \|(1 - \varphi^t)\nabla^2f\|_{L^2} +
  \langle t\rangle^{-1}\|1_{{\rm supp}\varphi^t}\nabla f\|_{L^2}.
\end{eqnarray}
Here we use $1_{\Omega}$ to denote the characteristic function of
$\Omega$. Note that the weight in the definition of
$\mathcal{X}_{|\alpha|}(Y)$ is equivalent to $\langle t\rangle$ on
the support of $1 - \varphi^t(y)$. Hence, we have
\begin{eqnarray}\nonumber
 t|D\Gamma^\alpha Y| &\lesssim& \big(\|D\Gamma^\alpha
  Y\|_{L^2} + \|\langle t\rangle(1 - \varphi^t)\nabla D
  \Gamma^\alpha Y\|_{L^2}\big)\ln^{\frac{1}{2}}\big(e +
  t\big)\\\nonumber
&&+\ \|D\Gamma^\alpha Y\|_{L^2} + \|(1 - \varphi^t)
  \nabla^2D\Gamma^\alpha Y\|_{L^2} + \langle t\rangle^{-1}
  \|1_{{\rm supp}\varphi^t}\nabla D\Gamma^\alpha
  Y\|_{L^2}\\\nonumber
&\lesssim& \big(\mathcal{E}_{|\alpha| +  1}^{\frac{1}{2}} +
  \mathcal{X}_{|\alpha| +  3}^{\frac{1}{2}}\big)
  \ln^{\frac{1}{2}}\big(e + t\big).
\end{eqnarray}
This completes the proof of \eqref{C-11}.

It remains to prove \eqref{C-3}. Notice that $r \geq 1$ for $t \geq 4$.
 Similarly as proving \eqref{C-1}, we calculate that
\begin{eqnarray}\nonumber
&&r\langle t-r\rangle|f(r\omega)|^2 \lesssim r\langle
  t-r\rangle\sum_{|\beta| \leq 1}\int_{\mathbb{S}^1}
  |\Omega^\beta f(r\omega)|^2d\sigma\\\nonumber
&&= - \sum_{|\beta| \leq 1}\int_{\mathbb{S}^1}d\sigma \int_r^\infty
  r\partial_\rho[\langle t-\rho\rangle
  |\Omega^\beta f(\rho\omega)|^2] d\rho\\\nonumber
&&\lesssim \sum_{|\beta| \leq 1}\int_{\mathbb{S}^1}d\sigma
  \int_r^\infty [\langle t-\rho\rangle |\Omega^\beta f(\rho\omega)||\partial_r
  \Omega^\beta f(\rho\omega)|+  |\Omega^\beta f(\rho\omega)|^2]  \rho d\rho\\\nonumber
&&\lesssim \sum_{|\beta| \leq 1}\int_{\mathbb{S}^1}d\sigma
  \int_r^\infty [\langle t-\rho\rangle^2 |\partial_r
  \Omega^\beta f(\rho\omega)|^2 +
  |\Omega^\beta f(\rho\omega)|^2]  \rho d\rho\\\nonumber
&&= \sum_{|\beta| \leq 1}[\|\langle t-r\rangle \partial_r
  \Omega^\beta f\|_{L^2} + \|\Omega^\beta  f\|_{L^2}]^2.
\end{eqnarray}
Slightly changing the definition of $\varphi^t$ and then taking $f =
\varphi^tD\Gamma^\alpha Y$ in the above inequality, one has
\begin{eqnarray}\nonumber
r\langle t-r\rangle|\varphi^tD\Gamma^\alpha Y|^2 &\lesssim&
  \sum_{|\beta| \leq 1}\big(\|\langle t-r\rangle \partial_r
  \Omega^\beta[\varphi^t D\Gamma^\alpha Y]\|_{L^2}
  + \|\Omega^\beta D\Gamma^\alpha Y\|_{L^2}\big)^2\\\nonumber
&\lesssim& \sum_{|\beta| \leq 1}\|\varphi^t\langle t-r\rangle
  \partial_r\Omega^\beta[ D\Gamma^\alpha Y]\|_{L^2}^2\\\nonumber
&&+\ \sum_{|\beta| \leq 1}\|\partial_r\varphi^t\langle
  t-r\rangle\Omega^\beta[D\Gamma^\alpha Y]\|_{L^2}^2
  + \mathcal{E}_{|\alpha| +  2}
\end{eqnarray}
which yields \eqref{C-3}. Here we used the fact that $\Omega$
commutes with $\varphi^t$ due to the symmetry of $\varphi^t$.

\end{proof}

Now let us study the decay properties of the second order
derivatives of unknowns under $L^\infty$ norm. The following lemma
shows that away from the light cone, the second derivatives of
unknowns decay in time like $\langle t\rangle^{-1}$. But near the
light cone, the decay rate is only $\langle
t\rangle^{-\frac{1}{2}}$, with an extra factor $\langle t -
r\rangle^{- 1}$. We emphasis that the 3D version has already appeared in \cite{KS96}.

\begin{lem}\label{DecayES}
Let $t \geq 4$. Then for $r \leq 2\langle t\rangle /3$, or $r \geq
5\langle t\rangle/4$, there hold
\begin{equation}\label{C-12}
 t |D^2\Gamma^\alpha Y| \lesssim \mathcal{X}_{|\alpha|
+ 4}^{\frac{1}{2}}.
\end{equation}
For $r \leq \frac{5\langle t\rangle}{2}$, there holds
\begin{equation}\label{C-4}
\langle r \rangle^{\frac{1}{2}}\langle t-r\rangle|D^2\Gamma^\alpha
Y| \lesssim \mathcal{X}_{|\alpha| + 4}^{\frac{1}{2}} +
\mathcal{E}_{|\alpha| + 3}^{\frac{1}{2}},
\end{equation}
\end{lem}
\begin{proof}
Let us use the cutoff function $\varphi^t$ in Lemma \ref{DecayEF}.
Note that the weight in the definition of $\mathcal{X}_{|\alpha| +
4}^{\frac{1}{2}}(Y)$ is equivalent to $\langle t\rangle$ on the
support of $1 - \varphi^t(x)$. Thus, by applying the simple Sobolev
imbedding $H^2(\mathbb{R}^2)\hookrightarrow L^\infty(\mathbb{R}^2)$,
we have
\begin{eqnarray}\nonumber
 t|D^2\Gamma^\alpha Y| &\lesssim&  t
  \|\big(1 - \varphi^t(y)\big)D^2\Gamma^\alpha Y\|_{L^2}\\\nonumber
&&+\  t  \|\big(1 -
  \varphi^t(y)\big)\nabla^2 D^2\Gamma^\alpha Y\|_{L^2}\\\nonumber
&&+\  t  \|\nabla\big(1 -
  \varphi^t(y)\big)\nabla D^2\Gamma^\alpha Y\|_{L^2}\\\nonumber
&&+\  t  \|\nabla^2 \big(1 -
  \varphi^t(y)\big)D^2\Gamma^\alpha Y\|_{L^2}\\\nonumber
&\lesssim& \mathcal{X}_{|\alpha| + 4}^{\frac{1}{2}} +
  \langle t\rangle^{- 1}\mathcal{X}_{|\alpha| +
  3}^{\frac{1}{2}} + \langle t\rangle^{- 2}\mathcal{X}_{|\alpha| +
  2}^{\frac{1}{2}}\\\nonumber
&\lesssim& \mathcal{X}_{|\alpha| + 4}^{\frac{1}{2}}.
\end{eqnarray}
This proves \eqref{C-12}.

Next, let us prove \eqref{C-4}. For $r \leq 1$, \eqref{C-4} is an
immediate consequence of \eqref{C-12}. We consider the case when $r
\geq 1$. Similarly as in proving \eqref{C-3}, one has
\begin{eqnarray}\nonumber
&&r\langle t-r\rangle^2|f(r\omega)|^2 \lesssim r\langle
  t-r\rangle^2\sum_{|\beta| \leq 1}\int_{\mathbb{S}^1}
  |\Omega^\beta f(r\omega)|^2d\sigma\\\nonumber
&&= - \sum_{|\beta| \leq 1}\int_{\mathbb{S}^1}d\sigma \int_r^\infty
  r\partial_\rho
  [\langle t-\rho\rangle^2 |\Omega^\beta f(\rho\omega)|^2] d\rho\\\nonumber
&&\lesssim \sum_{|\beta| \leq 1}\int_{\mathbb{S}^1}d\sigma
  \int_r^\infty \big[\langle t-\rho\rangle^2 |\Omega^\beta f(\rho\omega)||\partial_r
  \Omega^\beta f(\rho)|+\langle t-\rho\rangle |\Omega^\beta f(\rho\omega)|^2\big]  \rho d\rho\\\nonumber
&&\lesssim \sum_{|\beta| \leq 1}\int_{\mathbb{S}^1}d\sigma
  \int_r^\infty \big[\langle t-\rho\rangle^2 |\partial_r
  \Omega^\beta f(\rho\omega)|^2+\langle t-\rho\rangle^{2}
  |\Omega^\beta f(\rho\omega)|^2\big]  \rho d\rho\\\nonumber
&&= \sum_{|\beta| \leq 1}\big[\|\langle t-r\rangle \partial_r
  \Omega^\beta f\|_{L^2} + \|\langle t - r\rangle\Omega^\beta  f\|_{L^2}\big]^2.
\end{eqnarray}
Now let us choose another cutoff function $\widetilde{\varphi} \in
C_0^\infty(\mathbb{R}^2)$ which is radial and satisfies
\begin{equation}\nonumber
\widetilde{\varphi} = \begin{cases}1,\quad {\rm if}\ r \leq \frac{5}{2}\\
0,\quad {\rm if}\  r > 3\end{cases},\quad
|\nabla\widetilde{\varphi}| \leq 3.
\end{equation}
For each fixed $t \geq 4$, let $\widetilde{\varphi}^t(y) =
\widetilde{\varphi}(y/\langle t\rangle)$. Clearly, one has
$$\widetilde{\varphi}^t(y) \equiv 1\ \ {\rm for}\ r \leq \frac{5\langle t \rangle}{2},\quad
\widetilde{\varphi}^t(y) \equiv 0\ \ {\rm for}\ r \geq 3\langle t
\rangle$$ and
$$|\nabla\widetilde{\varphi}^t(y)| \leq 3\langle t\rangle^{-1}.$$ Taking $f =
\widetilde{\varphi}^tD^2\Gamma^\alpha Y$, one has
\begin{eqnarray}\nonumber
r\langle t-r\rangle^2|\widetilde{\varphi}^tD^2\Gamma^\alpha Y|^2
&\lesssim& \sum_{|\beta| \leq 1}\|\langle
  t-r\rangle\widetilde{\varphi}^t \partial_r\Omega^\beta
  D^2\Gamma^\alpha Y\|_{L^2}^2\\\nonumber
&&+\ \sum_{|\beta| \leq 1}\|\langle t-r\rangle
  \partial_r\widetilde{\varphi}^t\Omega^\beta D^2
  \Gamma^\alpha Y\|_{L^2}^2\\\nonumber
&& +\ \sum_{|\beta| \leq 1}\|\langle t-
  r\rangle\widetilde{\varphi}^t\Omega^\beta
  D^2\Gamma^\alpha Y\|_{L^2}^2\\\nonumber
&\lesssim& \|\langle t-r\rangle\widetilde{\varphi}^t
  D^2\Gamma^{|\alpha| + 2}Y\|_{L^2}^2
  + \|1_{{\rm supp}\widetilde{\varphi}^t}
  D^2\Gamma^{|\alpha| + 1}Y\|_{L^2}^2,
\end{eqnarray}
which gives \eqref{C-4} for $r \geq 1$.
\end{proof}

The next lemma gives a preliminary estimate for the weighted $L^2$
generalized energy norm $\mathcal{X}_{\kappa}$.   We remark that the
definition of $\mathcal{X}_{\kappa}$ here is different from the
original one appeared in \cite{KS96} and \cite{Sideris00} where similar results are obtained in 3D. For a self-contained presentation, we still include the detailed proof below.
\begin{lem}\label{WE-1}
There holds
\begin{equation}\nonumber
\mathcal{X}_{2}^{\frac{1}{2}} \lesssim \mathcal{E}_2^{\frac{1}{2}} + \langle
t\rangle\|(\partial_t^2 - \Delta)Y\|_{L^2}.
\end{equation}
\end{lem}
\begin{proof}
First of all, one may use the decomposition for gradient operator
\begin{equation}\nonumber
\nabla = \omega\partial_r + \frac{\omega^\perp}{r}\partial_\theta
\end{equation}
and the expression of Laplacian in polar coordinate
\begin{equation}\nonumber
\Delta = \partial_r^2 + \frac{1}{r}\partial_r +
\frac{1}{r^2}\partial_\theta^2
\end{equation}
to derive that
\begin{eqnarray}\label{C-5}
|\Delta Y - \partial_r^2Y| \leq \frac{|\partial_rY|}{r} +
\frac{|\partial_\theta^2Y|}{r^2} \lesssim \frac{|\nabla Y| +
|\nabla\Omega Y|}{r}.
\end{eqnarray}
Let us further write
\begin{eqnarray}\nonumber
(t^2 - r^2)\Delta Y &=& - t^2(\partial_t^2 - \Delta) Y - r^2(\Delta
  Y - \partial_r^2Y) + t^2\partial_t^2Y - r^2\partial_r^2Y\\\nonumber
&=& - t^2(\partial_t^2 - \Delta) Y - r^2(\Delta
  Y - \partial_r^2Y)\\\nonumber
&&+\ (t\partial_t - r\partial_r)(t\partial_t +
  r\partial_r - 1)Y.
\end{eqnarray}
Hence, using \eqref{C-5}, one has
\begin{eqnarray}\label{C-6}
|(t - r)\Delta Y| &\lesssim& t|(\partial_t^2 - \Delta) Y| + r|\Delta
  Y - \partial_r^2Y| + |DSY| + |DY|\\\nonumber
&\lesssim& |D\Gamma Y| + |DY| + t|(\partial_t^2 - \Delta) Y|.
\end{eqnarray}
We
remark that in \eqref{C-6}, $r$ can be even larger than $2\langle
t\rangle$.

Next, using \eqref{C-6} and integration by parts, one immediately
has
\begin{eqnarray}\nonumber
&&\|(t - r)\partial_i\partial_j Y\|_{L^2}^2 =
  \int(t - r)^2\partial_i
  \partial_j Y\partial_i\partial_j Ydy\\\nonumber
&&= 2\int(t - r)\omega_i\partial_j
  Y\partial_i\partial_j Ydy - \int(t - r)^2\partial_j
  Y\Delta\partial_j Ydy\\\nonumber
&&= 2\int(t - r)\omega_i\partial_j
  Y\partial_i\partial_j Ydy - 2\int(t - r)\omega_j\partial_j
  Y\Delta Ydy\\\nonumber
&&\quad +\ \int(t - r)^2\Delta Y\Delta
  Ydy\\\nonumber
&&\leq 10\|\nabla Y\|_{L^2}^2  + \frac{1}{2}
  \|(t - r)\partial_i\partial_jY\|_{L^2}^2
  + \int(t - r)^2\Delta Y\Delta Ydy.
\end{eqnarray}
Using \eqref{C-6}, one obtains that
\begin{eqnarray}\label{C-7}
\|(t - r)\partial_i\partial_j Y\|_{L^2} \lesssim
\mathcal{E}_2^{\frac{1}{2}} + \langle t\rangle\|(\partial_t^2 -
\Delta)Y\|_{L^2}.
\end{eqnarray}

To estimate $|(t - r)\partial_{t}\nabla Y|$, let us first write
\begin{eqnarray}\nonumber
(t - r)\partial_t\partial_rY &=& - (\partial_t -
  \partial_r)(t\partial_t + r\partial_r - 1)Y
  + t\partial_t^2Y - r\partial_r^2Y\\\nonumber
&=& - (\partial_t - \partial_r)(t\partial_t +
  r\partial_r - 1)Y + t(\partial_t^2 - \Delta)Y\\\nonumber
&&+\ (t - r)\Delta Y + r(\Delta Y - \partial_r^2Y).
\end{eqnarray}
Then using \eqref{C-5} and \eqref{C-6}, one has
\begin{eqnarray}\label{C-10}
|(t - r)\partial_t\partial_rY| \lesssim |D\Gamma Y| + |DY| +
t|(\partial_t^2 - \Delta) Y|.
\end{eqnarray}
Consequently, we have
\begin{eqnarray}\nonumber
|(t - r)\partial_t\partial_j Y| &=& \big|(t - r)
  \partial_t\omega_j\partial_r Y + (t -
  r) r^{-1}\omega_j^\perp\partial_t\partial_\theta Y\big|\\\nonumber
&\leq& |(t - r)\omega_j\partial_t\partial_r Y| + \big|r^{-1}
  (t\partial_t + r\partial_r)\partial_\theta
  Y - \partial_\theta(\partial_t +
  \partial_r) Y\big|\\\nonumber
&\lesssim& |(t - r)\omega_j\partial_t\partial_r Y| +
  r^{-1}|\partial_\theta SY| + |\partial_\theta
  \partial_tY| + |\partial_\theta
  (\omega\cdot\nabla Y)|\\\nonumber
&\lesssim& |(t - r)\omega_j\partial_t\partial_r Y| +
  r^{-1}|\partial_\theta SY| +
  |\partial_\theta\partial_tY|\\\nonumber
&&\quad +\ | \omega^\perp\cdot\nabla Y| + |\omega\cdot
  \partial_\theta\nabla Y|\\\nonumber
&\lesssim& |D\Gamma Y| + |DY| + t|(\partial_t^2 - \Delta) Y|,
\end{eqnarray}
which gives that
\begin{eqnarray}\label{C-8}
\|(t - r)\partial_t\partial_j Y\|_{L^2} \lesssim
\mathcal{E}_2^{\frac{1}{2}} + \langle t\rangle\|(\partial_t^2 -
\Delta)Y\|_{L^2}.
\end{eqnarray}

At last, using \eqref{C-6}, one has
\begin{eqnarray}\nonumber
|(t - r)\partial_t^2 Y| &\leq& \big|(t - r)
  (\partial_t^2 - \Delta)Y\big| + |(t -
  r) \Delta Y|\\\nonumber
&\lesssim& |(t - r)
  (\partial_t^2 - \Delta)Y| + |D\Gamma Y| + |DY|
  + t|(\partial_t^2 - \Delta) Y|\\\nonumber
&\lesssim& |D\Gamma Y| + |DY|
  +  \langle t\rangle|(\partial_t^2 - \Delta) Y|
\end{eqnarray}
if $r \leq 2\langle t\rangle$, and
\begin{eqnarray}\nonumber
\langle t\rangle|\partial_t^2 Y| \leq \langle t\rangle|
  (\partial_t^2 - \Delta)Y\big| + \langle t\rangle|\Delta Y|
\end{eqnarray}
if $r > 2\langle t\rangle$. Hence, by \eqref{C-7}, we have
\begin{eqnarray}\label{C-9}
&&\int_{r \leq 2\langle t\rangle}\langle t - r\rangle^2|\partial_t^2
  Y|^2dy + \int_{r > 2\langle t\rangle}\langle t\rangle^2|\partial_t^2
  Y|^2dy\\\nonumber
&&\lesssim \langle t\rangle^2\|(\partial_t^2 - \Delta)Y\|_{L^2}^2 +
  \mathcal{E}_2^{\frac{1}{2}}(Y) + \int_{r > 2\langle t\rangle}
  \langle t\rangle^2|\Delta Y|^2dy\\\nonumber
&&\lesssim \langle t\rangle^2\|(\partial_t^2 - \Delta)Y\|_{L^2}^2 +
  \mathcal{E}_2.
\end{eqnarray}

Then the lemma follows from \eqref{C-7}, \eqref{C-8} and
\eqref{C-9}.

\end{proof}

At the end of this section, let us show the estimate for good derivatives
$\omega_j\partial_t + \partial_j$ (see some related results in \cite{LSZ13}).
\begin{lem}\label{GoodDeri}
For $\frac{\langle t\rangle}{3} \leq r \leq \frac{5\langle
t\rangle}{2}$, there holds
\begin{equation}\nonumber
\langle t\rangle|\omega_j\partial_tDY + \partial_jDY| \lesssim
|D Y| + |D\Gamma Y| + t|(\partial_t^2 - \Delta) Y|.
\end{equation}
\end{lem}
\begin{proof}
First, let us calculate that
\begin{eqnarray}\nonumber
&&t(\partial_t + \partial_r)(\partial_t - \partial_r)Y\\\nonumber
&&= t(\partial_t^2 - \Delta)Y + t(\Delta - \partial_r^2)Y\\\nonumber
&&= t(\partial_t^2 - \Delta)Y + \frac{t}{r}\big(\partial_rY +
  \frac{\partial_\theta^2 Y}{r}\big).
\end{eqnarray}
Consequently, we have
\begin{eqnarray}\label{G-1}
&&t|(\partial_t + \partial_r)(\partial_t -
\partial_r)Y|\\\nonumber
&&\leq t|(\partial_t^2 - \Delta)Y| + \frac{t}{r}\big(|\nabla Y| +
   \frac{|(x_i\partial_j - x_j\partial_i) \partial_\theta
   Y|}{r}\big)\\\nonumber
&&\lesssim t|(\partial_t^2 - \Delta)Y| + \frac{t}{r}\big(|\nabla Y|
  + |\nabla\Omega Y|\big).
\end{eqnarray}
Next, using \eqref{C-5}, \eqref{C-6}, \eqref{C-10} and \eqref{G-1},
we calculate that
\begin{eqnarray}\nonumber
&&t|(\partial_t + \partial_r)\partial_rY|\\\nonumber &&\leq
\frac{t}{2}|(\partial_t +
\partial_r)(\partial_t -
  \partial_r)Y| + \frac{t}{2}|(\partial_t + \partial_r)(\partial_t +
  \partial_r)Y|\\\nonumber
&&\lesssim \frac{t}{2}|(\partial_t + \partial_r)(\partial_t -
  \partial_r)Y| + \frac{1}{2}|S(\partial_t +
  \partial_r)Y| + \frac{1}{2}|(t - r)\partial_r(\partial_t +
  \partial_r)Y|\\\nonumber
&&\lesssim \frac{t}{2}|(\partial_t + \partial_r)(\partial_t -
  \partial_r)Y| + \frac{1}{2}|S(\partial_t +
  \partial_r)Y|\\\nonumber
&&\quad +\ \frac{1}{2}|(t - r)\partial_{tr}^2Y| + \frac{1}{2}|(t -
  r)(\Delta - \partial_r^2)Y| + \frac{1}{2}|(t - r)\Delta
  Y|\\\nonumber
&&\lesssim t|(\partial_t^2 - \Delta)Y| + (1 +
  \frac{t}{r})\big(|\nabla Y| + |\nabla\Gamma Y|\big).
\end{eqnarray}
Hence, we have
\begin{eqnarray}\label{G-2}
&&t|(\omega_j\partial_t + \partial_j)\partial_kY|\\\nonumber &&=
  \big|t\omega_j(\partial_t + \partial_r)\partial_kY +
  \frac{t}{r}\omega_j^\perp\partial_\theta
  \partial_kY\big|\\\nonumber
&&\lesssim |t\omega_j(\partial_t + \partial_r)\omega_k\partial_rY| + |\frac{t}{r}\frac{\partial_\theta}{r}\partial_r Y| +
  \frac{t}{r}\big(|\partial_\theta
  \partial_kY| + |\partial_r\partial_\theta Y|\big)\\\nonumber
&&\lesssim t|(\partial_t^2 - \Delta)Y| + (1 +
  \frac{t}{r})\big(|D Y| + |D\Gamma Y|\big).
\end{eqnarray}

At last, let us calculate that
\begin{eqnarray}\nonumber
t|(\partial_t + \partial_r)\partial_tY| &\leq& t|(\partial_t +
  \partial_r)(\partial_t - \partial_r)Y| + t|(\partial_t +
  \partial_r)\partial_rY|\\\nonumber
&\leq& t|(\partial_t +
  \partial_r)(\partial_t - \partial_r)Y| + t|\omega_j\omega_k
  (\omega_j\partial_t +
  \partial_j)\partial_kY|\\\nonumber
&&+\ t|\omega_j\partial_j\omega_k\partial_kY|
\end{eqnarray}
which together with \eqref{G-1} and \eqref{G-2} gives that
\begin{eqnarray}\label{G-3}
&&t|(\omega_j\partial_t + \partial_j)\partial_tY|\\\nonumber &&\leq
  t\big|(\partial_t + \partial_r)\partial_tY| +
  |r^{-1}\partial_\theta\partial_tY\big|\\\nonumber
&&\lesssim t|(\partial_t^2 - \Delta)Y| + (1 +
  \frac{t}{r})\big(|\nabla Y| + |\nabla\Gamma Y|\big).
\end{eqnarray}
Then the lemma follows from \eqref{G-2} and \eqref{G-3}.
\end{proof}

\section{Estimate of the $L^2$ Weighted Norm}\label{WE}

Now we are going to estimate the $L^2$ weighted generalized energy
$\mathcal{X}_{\kappa}$. First of all, we prove the following lemma
which says that the $L^2$ norm of $\nabla\Gamma^\alpha p$ and
$\|(\partial_t^2 - \Delta)\Gamma^\alpha Y\|_{L^2}$ which involve the
$(|\alpha| + 2)$-th order derivatives of unknowns can be bounded by
certain matters which only involve $(|\alpha| + 1)$-th order
derivatives of unknowns. This surprising result is based on the
inherent special structures of nonlinearities in the system.

\begin{lem}\label{SN-1}
Suppose that $\|\nabla Y\|_{L^\infty} \leq \delta$ for some
absolutely positive constant $\delta < 1$. Then there holds
\begin{equation}\nonumber
\|\nabla\Gamma^\alpha p\|_{L^2} + \|(\partial_t^2 -
\Delta)\Gamma^\alpha Y\|_{L^2} \lesssim \Pi(|\alpha| + 2)
\end{equation}
provided that $\delta$ is small enough, where
\begin{eqnarray}\label{D-9}
&&\Pi(|\alpha| + 2) \lesssim \sum_{\beta + \gamma = \alpha,\
  \gamma \neq \alpha}\big\||\nabla\Gamma^\beta Y||(\partial_t^2
   - \Delta)\Gamma^\gamma Y|\big\|_{L^2}\\\nonumber
&&\quad +\ \sum_{\beta + \gamma = \alpha,
  |\beta| > |\gamma|}\Pi_1 + \sum_{\beta + \gamma = \alpha,
  |\beta| \geq |\gamma|}\Pi_2,
\end{eqnarray}
where $\Pi_1$ and $\Pi_2$ are given by
\begin{eqnarray}\label{D-11}
\Pi_1 = \big\|(- \Delta)^{- \frac{1}{2}}\nabla\cdot\big(\partial_t
\Gamma^\beta Y^1\nabla^\perp\partial_t\Gamma^\gamma Y^2 -
\partial_j\Gamma^\beta Y^1\nabla^\perp\partial_j \Gamma^\gamma
Y^2\big)\big\|_{L^2},
\end{eqnarray}
and
\begin{eqnarray}\label{D-12}
\Pi_2 = \big\|(- \Delta)^{-
\frac{1}{2}}\nabla\cdot\big(\partial_t\Gamma^\beta Y^2
\nabla^\perp\partial_t\Gamma^\gamma Y^1 - \partial_j\Gamma^\beta
Y^2\nabla^\perp
\partial_j\Gamma^\gamma  Y^1\big)\big\|_{L^2}.
\end{eqnarray}
\end{lem}
\begin{proof}
By \eqref{Elasticity-D}, one has
\begin{eqnarray}\nonumber
- \nabla\Gamma^\alpha p = (\nabla X)^{\top}(\partial_t^2 - \Delta)
\Gamma^\alpha Y + \sum_{\beta + \gamma = \alpha,\ \gamma \neq
\alpha}C_\alpha^\beta(\nabla \Gamma^\beta Y)^{\top}(\partial_t^2 -
\Delta)\Gamma^\gamma Y.
\end{eqnarray}
Applying the divergence operator to the above equation and then
applying the operator $\nabla(- \Delta)^{-1}$, we obtain that
\begin{eqnarray}\nonumber
\nabla\Gamma^\alpha p &=& \sum_{\beta + \gamma = \alpha,\
  \gamma \neq \alpha}C_\alpha^\beta\big\{\nabla(- \Delta)^{-1}\nabla\cdot[(\nabla\Gamma^\beta Y)^\top(\partial_t^2
   - \Delta)\Gamma^\gamma Y)]\\\nonumber
&&+\ \nabla(- \Delta)^{-1}\nabla\cdot[(\nabla Y)^\top(\partial_t^2 -
  \Delta)\Gamma^\alpha Y]\\\nonumber
&&+\ \nabla(- \Delta)^{-1}\nabla\cdot[(\partial_t^2 -
  \Delta)\Gamma^\alpha Y]\big\}.
\end{eqnarray}
Hence, using the fact that the Riesz operator is bounded in $L^2$,
one has
\begin{eqnarray}\label{D-1}
\|\nabla\Gamma^\alpha p\|_{L^2} &\lesssim& \sum_{\beta + \gamma = \alpha,\
  \gamma \neq \alpha}\|(\nabla\Gamma^\beta Y)^\top(\partial_t^2
   - \Delta)\Gamma^\gamma Y)\|_{L^2}\\\nonumber
&&+\ \|(\nabla Y)^\top(\partial_t^2 - \Delta)\Gamma^\alpha Y\|_{L^2}
  + \|\nabla(- \Delta)^{-1}\nabla\cdot(\partial_t^2
  - \Delta)\Gamma^\alpha Y\|_{L^2}.
\end{eqnarray}
Here we kept the Riesz operator in the last quantity on the right
hand side of the above estimate, which needs further treatments
using the fantastic inherent structures of the system.

First of all, using \eqref{Struc-2}, we compute that
\begin{eqnarray}\label{D-2}
&&\nabla\cdot (\partial_t^2
  - \Delta)\Gamma^\alpha Y\\\nonumber
&&= (\partial_t^2 - \Delta)\sum_{\beta + \gamma = \alpha}C_\alpha^\beta
  \big[\partial_1\Gamma^\beta  Y^2\partial_2\Gamma^\gamma  Y^1
  - \partial_1\Gamma^\gamma  Y^1\partial_2\Gamma^\beta  Y^2\big]\\\nonumber
&&= \sum_{\beta + \gamma = \alpha}C_\alpha^\beta\big\{
  \big[\partial_1(\partial_t^2 - \Delta)\Gamma^\beta  Y^2\partial_2\Gamma^\gamma  Y^1
  - \partial_1\Gamma^\gamma  Y^1\partial_2(\partial_t^2 - \Delta)\Gamma^\beta  Y^2\big]\\\nonumber
&&\quad +\ \sum_{\beta + \gamma = \alpha}
  \big[\partial_1\Gamma^\beta  Y^2\partial_2(\partial_t^2 - \Delta)\Gamma^\gamma  Y^1
  - \partial_1(\partial_t^2 - \Delta)\Gamma^\gamma  Y^1\partial_2\Gamma^\beta  Y^2\big]\\\nonumber
&&\quad +\ 2\sum_{\beta + \gamma = \alpha}
  \big[\partial_1\partial_t\Gamma^\beta  Y^2\partial_2\partial_t\Gamma^\gamma  Y^1
  - \partial_1\partial_t\Gamma^\gamma  Y^1\partial_2\partial_t\Gamma^\beta
  Y^2\big]\\\nonumber
&&\quad -\ 2\sum_{\beta + \gamma = \alpha}\big[\partial_1\partial_j
  \Gamma^\beta  Y^2\partial_2\partial_j\Gamma^\gamma  Y^1
  - \partial_1\partial_j\Gamma^\gamma  Y^1\partial_2
  \partial_j\Gamma^\beta  Y^2\big]\big\}.
\end{eqnarray}
Noting the inherent cancellation relation, one sees that the first
two terms on the right hand side of \eqref{D-2} can be reorganized
to be
\begin{eqnarray}\label{D-3}
\nabla\cdot\sum_{\beta + \gamma = \alpha}C_\alpha^\beta \big[(\partial_t^2 -
\Delta)\Gamma^\gamma  Y^1\nabla^\perp\Gamma^\beta  Y^2 -
(\partial_t^2 - \Delta)\Gamma^\beta  Y^2\nabla^\perp\Gamma^\gamma
Y^1\big].
\end{eqnarray}

We still need to take care of the last two terms on the right hand
side of \eqref{D-2}. We first divide them into three parts:
$$A_{11} + A_{12} + A_{13},$$
where $A_{11}$ in the portion of the summation in which $|\beta| =
|\gamma|$:
\begin{eqnarray}\label{D-4}
A_{11} &=& 2\sum_{\beta + \gamma = \alpha, |\beta| = |\gamma|}C_\alpha^\beta
  \big[\partial_1\partial_t\Gamma^\beta  Y^2\partial_2\partial_t\Gamma^\gamma  Y^1
  - \partial_1\partial_t\Gamma^\gamma  Y^1\partial_2\partial_t\Gamma^\beta
  Y^2\big]\\\nonumber
&& -\ 2\sum_{\beta + \gamma = \alpha, |\beta| = |\gamma|}C_\alpha^\beta
  \big[\partial_1\partial_j\Gamma^\beta  Y^2\partial_2\partial_j
  \Gamma^\gamma  Y^1 - \partial_1\partial_j\Gamma^\gamma  Y^1
  \partial_2\partial_j\Gamma^\beta  Y^2\big]\\\nonumber
&=& - 2\nabla\cdot\sum_{\beta + \gamma = \alpha, |\beta| = |\gamma|}C_\alpha^\beta
  \big[\partial_t\Gamma^\beta  Y^2\nabla^\perp\partial_t\Gamma^\gamma
  Y^1 - \partial_j\Gamma^\beta  Y^2\nabla^\perp\partial_j
  \Gamma^\gamma  Y^1\big],
\end{eqnarray}
$A_{12}$ is responsible for the terms involving time derivatives in
the summation when $|\beta| \neq |\gamma|$:
\begin{eqnarray}\nonumber
A_{12} &=& 2\Big(\sum_{\beta + \gamma = \alpha, |\beta| > |\gamma|}
   + \sum_{\beta + \gamma = \alpha, |\beta| < |\gamma|}\Big)C_\alpha^\beta
   \partial_1\partial_t\Gamma^\beta  Y^2\partial_2\partial_t
   \Gamma^\gamma Y^1\\\nonumber
&&-\ 2\Big(\sum_{\beta + \gamma = \alpha, |\beta| >
  |\gamma|} + \sum_{\beta + \gamma = \alpha, |\beta|
  < |\gamma|}\Big)C_\alpha^\beta\partial_1\partial_t\Gamma^\gamma
  Y^1\partial_2\partial_t\Gamma^\beta Y^2,
\end{eqnarray}
and $A_{13}$ is the portion in the summation when $|\beta| \neq
|\gamma|$ which only involves spatial derivatives:
\begin{eqnarray}\nonumber
A_{13} &=& 2\Big(\sum_{\beta + \gamma = \alpha, |\beta| > |\gamma|}
  + \sum_{\beta + \gamma = \alpha, |\beta| < |\gamma|}\Big)C_\alpha^\beta
  \partial_1\partial_j\Gamma^\gamma  Y^1\partial_2\partial_j\Gamma^\beta  Y^2 \\\nonumber
&&-\ 2\Big(\sum_{\beta + \gamma = \alpha, |\beta| > |\gamma|}
  + \sum_{\beta + \gamma = \alpha, |\beta| < |\gamma|}\Big)C_\alpha^\beta\partial_1
  \partial_j\Gamma^\beta  Y^2\partial_2\partial_j\Gamma^\gamma  Y^1.
\end{eqnarray}
Note that if $|\alpha|$ is odd, then $A_{11} = 0$.

By symmetry of $\beta$ and $\gamma$, we can rewrite $A_{12}$ and
$A_{13}$ as
\begin{eqnarray}\nonumber
A_{12} &=& 2\sum_{\beta + \gamma = \alpha, |\beta| > |\gamma|}C_\alpha^\beta
   \big(\partial_1\partial_t\Gamma^\beta  Y^2\partial_2\partial_t
   \Gamma^\gamma Y^1 + \partial_1\partial_t\Gamma^\gamma  Y^2\partial_2\partial_t
   \Gamma^\beta Y^1\big)\\\nonumber
&&-\  2\sum_{\beta + \gamma = \alpha, |\beta| >
  |\gamma|}C_\alpha^\beta\big(\partial_1\partial_t\Gamma^\gamma
  Y^1\partial_2\partial_t\Gamma^\beta Y^2 + \partial_1\partial_t\Gamma^\beta
  Y^1\partial_2\partial_t\Gamma^\gamma Y^2\big),
\end{eqnarray}
and
\begin{eqnarray}\nonumber
A_{13} &=& 2\sum_{\beta + \gamma = \alpha, |\beta|
  > |\gamma|}C_\alpha^\beta\big(\partial_1\partial_j\Gamma^\gamma  Y^1\partial_2
  \partial_j\Gamma^\beta Y^2 + \partial_1\partial_j\Gamma^\beta
  Y^1\partial_2\partial_j\Gamma^\gamma Y^2\big)\\\nonumber
&&-\ 2\sum_{\beta + \gamma = \alpha, |\beta| > |\gamma|}C_\alpha^\beta
  \big(\partial_1\partial_j\Gamma^\beta  Y^2\partial_2
  \partial_j\Gamma^\gamma  Y^1 + \partial_1\partial_j\Gamma^\gamma
  Y^2\partial_2\partial_j\Gamma^\beta Y^1\big).
\end{eqnarray}
By merging the first and third terms, the second and the last terms
respectively, we further rewrite $A_{12}$ and $A_{13}$ as follows:
\begin{eqnarray}\nonumber
A_{12} = 2\nabla\cdot\sum_{\beta + \gamma = \alpha, |\beta| >
  |\gamma|}C_\alpha^\beta\big( - \partial_t
   \Gamma^\beta Y^2\nabla^\perp\partial_t\Gamma^\gamma Y^1 + \partial_t
   \Gamma^\beta Y^1\nabla^\perp\partial_t\Gamma^\gamma Y^2\big),
\end{eqnarray}
and
\begin{eqnarray}\nonumber
A_{13} = 2\nabla\cdot\sum_{\beta + \gamma = \alpha, |\beta|
  > |\gamma|}C_\alpha^\beta\big(\partial_j\Gamma^\beta Y^2\nabla^\perp\partial_j\Gamma^\gamma  Y^1
   - \partial_j\Gamma^\beta Y^1\nabla^\perp\partial_j\Gamma^\gamma Y^2\big).
\end{eqnarray}
Now it is clear that we may add up the above two identities and
figure out the contribution of $A_{12}$ and $A_{13}$ to \eqref{D-3},
which is
\begin{eqnarray}\label{D-5}
&&A_{12} + A_{13}\\\nonumber &&= 2\nabla\cdot\sum_{\beta + \gamma =
  \alpha,  |\beta| > |\gamma|}C_\alpha^\beta\big[\big(\partial_t
   \Gamma^\beta Y^1\nabla^\perp\partial_t\Gamma^\gamma Y^2
   - \partial_j\Gamma^\beta Y^1\nabla^\perp\partial_j
   \Gamma^\gamma Y^2\big)\\\nonumber
&&\quad + \big(\partial_j\Gamma^\beta Y^2\nabla^\perp
  \partial_j\Gamma^\gamma  Y^1 - \partial_t\Gamma^\beta Y^2
   \nabla^\perp\partial_t\Gamma^\gamma Y^1\big)\big].
\end{eqnarray}

Let us insert \eqref{D-3}, \eqref{D-4} and \eqref{D-5} into
\eqref{D-2} to derive that
\begin{eqnarray}\label{D-7}
&&\|\nabla(-\Delta)\nabla\cdot (\partial_t^2
  - \Delta)\Gamma^\alpha Y\|_{L^2}\\\nonumber
&&\lesssim \sum_{\beta + \gamma = \alpha}\big\|(\partial_t^2 -
  \Delta)\Gamma^\gamma Y^1\nabla^\perp\Gamma^\beta  Y^2 -
  (\partial_t^2 - \Delta)\Gamma^\beta  Y^2\nabla^\perp
  \Gamma^\gamma Y^1\big\|_{L^2}\\\nonumber
&&\quad +\ \sum_{\beta + \gamma = \alpha,
  |\beta| > |\gamma|}\Pi_1 + \sum_{\beta + \gamma = \alpha,
  |\beta| \geq |\gamma|}\Pi_2.
\end{eqnarray}
Here $\Pi_1$ and $\Pi_2$ are given in \eqref{D-11} and \eqref{D-12}.
We emphasis that in the expressions for $\Pi_1$ and $\Pi_2$ we still
kept the zero order Riesz operator. A crude estimate by removing them
directly is not enough to take the full advantage of the structure
of the system, which may only lead to an almost global existence
result and recovers what we already proved in \cite{LSZ13} by a
different method. Now let us insert \eqref{D-7} into \eqref{D-1} to
derive that
\begin{eqnarray}\label{D-8}
\|\nabla\Gamma^\alpha p\|_{L^2} \lesssim \|\nabla
Y\|_{L^\infty}\|(\partial_t^2 - \Delta)\Gamma^\alpha Y\|_{L^2} +
\Pi(|\alpha| + 2),
\end{eqnarray}
where $\Pi(|\alpha| + 2)$ is given in \eqref{D-9}. Using the
equation \eqref{Elasticity-D} and \eqref{D-8}, one has
\begin{eqnarray}\nonumber
&&\|(\partial_t^2 - \Delta)\Gamma^\alpha Y\|_{L^2}\\\nonumber
&&\lesssim \|(\nabla X)^{-T}\|_{L^\infty}\big(\|\nabla\Gamma^\alpha p\|_{L^2} + \sum_{\beta + \gamma
  = \alpha,\ \gamma \neq \alpha}\|\|(\nabla
  \Gamma^\beta Y)^{\top}(\partial_t^2 - \Delta)\Gamma^\gamma
  Y\|_{L^2}\big)\\\nonumber
&&\lesssim \|\nabla
  Y\|_{L^\infty}\|(\partial_t^2 - \Delta)\Gamma^\alpha
  Y\|_{L^2} + \Pi(|\alpha| + 2),
\end{eqnarray}
which gives that
\begin{eqnarray}\label{D-10}
\|(\partial_t^2 - \Delta)\Gamma^\alpha Y\|_{L^2} \lesssim
\Pi(|\alpha| + 2)
\end{eqnarray}
provided that $\|\nabla Y\|_{L^\infty}$ is appropriately smaller
than an absolute positive constant $\delta < 1$. Inserting
\eqref{D-10} into \eqref{D-8}, one also has
\begin{eqnarray}\nonumber
\|\nabla\Gamma^\alpha p\|_{L^2} \lesssim \Pi(|\alpha| + 2).
\end{eqnarray}
We have proved the lemma.
\end{proof}

In the next lemma, we will use Lemma \ref{SN-1} to estimate the main
source of nonlinearities in \eqref{Elasticity-D} by carefully
dealing with the last two terms in \eqref{D-9}. On the right hand
sides of those estimates that we are going to prove, we did not gain
derivatives since the both sides are of the same order. But what we
gain is the time decay rate. In section \ref{HEE} when we performing
higher order energy estimate, we will deal with the last two terms
in \eqref{D-9} once again, in a different way. The purpose there is
to gain one derivative, with the price of slowing down the decay
rate in time.
\begin{lem}\label{SN-2}
Suppose that $\kappa \geq 10$. There exists $\delta > 0$ such that
if $\mathcal{E}_{\kappa - 2} \leq \delta$, then there hold
\begin{equation}\nonumber
\langle t\rangle\|\nabla\Gamma^{\leq\kappa - 4} p\|_{L^2} + \langle
t\rangle\|(\partial_t^2 - \Delta)\Gamma^{\leq\kappa - 4} Y\|_{L^2}
\lesssim \mathcal{E}_{ \kappa - 2}^{\frac{1}{2}}\big(\mathcal{E}_{
\kappa - 2}^{\frac{1}{2}} + \mathcal{X}_{\kappa -
2}^{\frac{1}{2}}\big)
\end{equation}
and
\begin{equation}\nonumber
\langle t\rangle\|\nabla\Gamma^{\leq\kappa - 2} p\|_{L^2} + \langle
t\rangle\|(\partial_t^2 - \Delta)\Gamma^{\leq\kappa - 2} Y\|_{L^2}
\lesssim \mathcal{E}_{\kappa}^{\frac{1}{2}}\big(\mathcal{E}_{ \kappa
- 2}^{\frac{1}{2}} + \mathcal{X}_{ \kappa - 2}^{\frac{1}{2}}\big).
\end{equation}
\end{lem}
\begin{proof}
We need deal with the two quantities in \eqref{D-11} and
\eqref{D-12}. They can be estimated in a similar way. Below we only
present the estimate for $\Pi_1$ in \eqref{D-11}. We first deal with
the integrals away from the light cone. Let $\phi^t$ be defined in
Lemma \ref{DecayEF}. It is easy to have the following first step
estimate
\begin{eqnarray}\nonumber
&&\sum_{\beta + \gamma = \alpha,
  |\gamma| \leq [|\alpha|/2]}\Big(\big\|\big(|\partial_t
   \Gamma^\beta Y^1\nabla^\perp\partial_t\Gamma^\gamma Y^2
   - \partial_j\Gamma^\beta Y^1\nabla^\perp\partial_j
   \Gamma^\gamma Y^2|\\\nonumber
&&\quad\quad\quad\quad +\ |\partial_t\Gamma^\beta Y^2
  \nabla^\perp\partial_t\Gamma^\gamma Y^1 - \partial_j\Gamma^\beta Y^2\nabla^\perp
  \partial_j\Gamma^\gamma  Y^1|\big)(1 - \phi^t)\big\|_{L^2}\\\nonumber
&&\quad\quad\ \ \lesssim \sum_{\beta + \gamma = \alpha,
  |\gamma| \leq [|\alpha|/2]}\|D\Gamma^\beta Y\|_{L^2}
  \|1_{{\rm supp}(1 - \varphi^t)}D^2\Gamma^\gamma Y\|_{L^\infty}.
\end{eqnarray}
Using Lemma \ref{DecayES}, the above is bounded by
\begin{eqnarray}\nonumber
\langle t\rangle^{-1}\mathcal{E}_{ |\alpha| +
1}^{\frac{1}{2}}\mathcal{X}_{[|\alpha|/2] + 4}^{\frac{1}{2}}.
\end{eqnarray}
Hence the estimate \eqref{D-9} is improved to be:
\begin{eqnarray}\label{F-1}
&&\Pi(|\alpha| + 2) \lesssim \langle
  t\rangle^{-1}\mathcal{E}_{ |\alpha| +
  1}^{\frac{1}{2}}\mathcal{X}_{[|\alpha|/2]
  + 4}^{\frac{1}{2}}\\\nonumber
&&\quad +\ \sum_{\beta + \gamma = \alpha,\
  \gamma \neq \alpha}\big\||\nabla\Gamma^\beta Y||(\partial_t^2
   - \Delta)\Gamma^\gamma Y|\big\|_{L^2}\\\nonumber
&&\quad +\ \sum_{\beta + \gamma = \alpha,
  |\gamma| \leq [|\alpha|/2]}\big(\Pi_1(\varphi^t) + \Pi_2(\varphi^t)\big),
\end{eqnarray}
where
\begin{eqnarray}\nonumber
\Pi_1(\varphi^t)  = \big\|(- \Delta)^{-
  \frac{1}{2}}\nabla\cdot\big[\varphi^t\big(\partial_t
   \Gamma^\beta Y^1\nabla^\perp\partial_t
   \Gamma^\gamma Y^2 -
  \partial_j\Gamma^\beta Y^1\nabla^\perp\partial_j
  \Gamma^\gamma Y^2\big)\big]\big\|_{L^2}
\end{eqnarray}
and
\begin{eqnarray}\nonumber
\Pi_2(\varphi^t) = \big\|(- \Delta)^{-
\frac{1}{2}}\nabla\cdot\big[\varphi^t\big(\partial_t\Gamma^\beta Y^2
\nabla^\perp\partial_t\Gamma^\gamma Y^1 - \partial_j\Gamma^\beta
Y^2\nabla^\perp
\partial_j\Gamma^\gamma  Y^1\big)\big]\big\|_{L^2}.
\end{eqnarray}
We will still need the use of this estimate in section \ref{HEE}.

Now let us deal with the third line of \eqref{F-1}. First of all, we
have
\begin{eqnarray}\nonumber
&&\nabla\cdot\big(\varphi^t\partial_t\Gamma^\beta Y^1
  \nabla^\perp\partial_t\Gamma^\gamma Y^2 - \varphi^t\partial_j\Gamma^\beta
  Y^1\nabla^\perp\partial_j\Gamma^\gamma Y^2\big)\\\nonumber
&&= \nabla\cdot\big(\varphi^t\partial_t\Gamma^\beta Y^1
  \nabla^\perp[\omega_j(\omega_j\partial_t + \partial_j)\Gamma^\gamma Y^2]\\\nonumber
&&\quad -\ \varphi^t\partial_j\Gamma^\beta
  Y^1\nabla^\perp(\omega_j\partial_t + \partial_j)\Gamma^\gamma Y^2\big)\\\nonumber
&&\quad +\ \nabla\cdot\big(\varphi^t\partial_j\Gamma^\beta
  Y^1\nabla^\perp(\omega_j\partial_t\Gamma^\gamma Y^2) - \varphi^t\partial_t\Gamma^\beta Y^1
  \nabla^\perp[\omega_j\partial_j\Gamma^\gamma Y^2]\big).
\end{eqnarray}
The last line on the right hand side of the above equality can be
rewritten as
\begin{eqnarray}\nonumber
\nabla^\perp\cdot\big(\omega_j\partial_t\Gamma^\gamma
Y^2\nabla(\varphi^t\partial_j\Gamma^\beta Y^1) -
\omega_j\partial_j\Gamma^\gamma
Y^2\nabla(\varphi^t\partial_t\Gamma^\beta Y^1)\big),
\end{eqnarray}
which can be further re-organized as follows:
\begin{eqnarray}\nonumber
&&\nabla^\perp\cdot\big((\omega_j\partial_t +
  \partial_j)\Gamma^\gamma Y^2\nabla(\varphi^t\partial_j
  \Gamma^\beta Y^1) - \partial_j\Gamma^\gamma Y^2\nabla\varphi^t
  (\omega_j\partial_t + \partial_j)\Gamma^\beta Y^1\\\nonumber
&&\quad - \partial_j\Gamma^\gamma Y^2\varphi^t(\omega_j\partial_t +
  \partial_j)\nabla\Gamma^\beta Y^1\big)\\\nonumber
&&= \nabla\cdot\big(\varphi^t\partial_j
  \Gamma^\beta Y^1\nabla^\perp(\omega_j\partial_t +
  \partial_j)\Gamma^\gamma Y^2\big)\\\nonumber
&&\quad - \nabla^\perp\cdot\big(\partial_j\Gamma^\gamma
  Y^2\varphi^t(\omega_j\partial_t +
  \partial_j)\nabla\Gamma^\beta Y^1\big)\\\nonumber
&&\quad - \nabla^\perp\cdot\big(\partial_j\Gamma^\gamma
  Y^2\nabla\varphi^t(\omega_j\partial_t + \partial_j)\Gamma^\beta Y^1\big).
\end{eqnarray}
Consequently, we have
\begin{eqnarray}\label{F-5}
&&\nabla\cdot\big(\varphi^t\partial_t\Gamma^\beta Y^1
  \nabla^\perp\partial_t\Gamma^\gamma Y^2 - \varphi^t\partial_j\Gamma^\beta
  Y^1\nabla^\perp\partial_j\Gamma^\gamma Y^2\big)\\\nonumber
&&= \nabla\cdot\big(\varphi^t\partial_t\Gamma^\beta Y^1
  \nabla^\perp[\omega_j(\omega_j\partial_t + \partial_j)\Gamma^\gamma Y^2]\\\nonumber
&&\quad -\ \varphi^t\partial_j\Gamma^\beta
  Y^1\nabla^\perp(\omega_j\partial_t + \partial_j)\Gamma^\gamma Y^2
  + \varphi^t\partial_j
  \Gamma^\beta Y^1\nabla^\perp(\omega_j\partial_t +
  \partial_j)\Gamma^\gamma Y^2\big)\\\nonumber
&&\quad -\ \nabla^\perp\cdot\big(\partial_j\Gamma^\gamma
  Y^2\varphi^t(\omega_j\partial_t + \partial_j)\nabla
  \Gamma^\beta Y^1 + \partial_j\Gamma^\gamma Y^2\nabla\varphi^t
  (\omega_j\partial_t + \partial_j)\Gamma^\beta Y^1\big).
\end{eqnarray}
The expression in \eqref{F-5} will be rewritten as a different form
in section \ref{HEE} for a different purpose.

Now we are ready to estimate the third line on the right hand side
of \eqref{F-1} as follows:
\begin{eqnarray}\nonumber
&&\big\|(- \Delta)^{-
  \frac{1}{2}}\nabla\cdot\big[\varphi^t\big(\partial_t
   \Gamma^\beta Y^1\nabla^\perp\partial_t\Gamma^\gamma
   Y^2 - \partial_j\Gamma^\beta
  Y^1\nabla^\perp\partial_j\Gamma^\gamma Y^2\big)
  \big]\big\|_{L^2}\\\nonumber
&&\lesssim \|D\Gamma^\beta Y\|_{L^2}\|1_{{\rm supp}\varphi^t}
  (\omega_j\partial_t + \partial_j)\nabla\Gamma^\gamma
  Y\|_{L^\infty}\\\nonumber
&&\quad +\ \|1_{{\rm supp}\varphi^t}D\Gamma^\gamma
Y\|_{L^\infty}\|1_{{\rm supp}\varphi^t}
  (\omega_j\partial_t + \partial_j)\nabla\Gamma^\beta
  Y\|_{L^2}\\\nonumber
&&\quad +\ \langle t\rangle^{-1}\|1_{{\rm
  supp}\varphi^t}D\Gamma^\gamma
  Y\|_{L^\infty}\|D\Gamma^\beta Y\|_{L^2}.
\end{eqnarray}
Note that the last term in the above is due to the commutation between $\nabla$ and good derivatives.  We first use Lemma \ref{GoodDeri} to bound the quantities on the
right hand side of the above estimate by
\begin{eqnarray}\nonumber
&&\langle t\rangle^{-1}\|D\Gamma^\beta Y\|_{L^2}\big(
  \|1_{{\rm supp}\varphi^t}\langle t\rangle (\partial_t^2 - \Delta)
  \Gamma^\gamma Y\|_{L^\infty} + \|1_{{\rm supp}\varphi^t}D\Gamma^{
  \leq|\gamma| + 1} Y\|_{L^\infty}\big)\\\nonumber
&&\quad +\ \langle t\rangle^{-1}\|1_{{\rm
  supp}\varphi^t}D\Gamma^\gamma Y\|_{L^\infty}
  \big(\|\langle t\rangle(\partial_t^2
  - \Delta)\Gamma^\beta Y\|_{L^2} + \|D\Gamma^{\leq|\beta| + 1} Y\|_{L^2}\big)\\\nonumber
&&\quad +\ \langle t\rangle^{-1}\|1_{{\rm
  supp}\varphi^t}D\Gamma^\gamma Y\|_{L^\infty}\|
  D\Gamma^\beta Y\|_{L^2}.
\end{eqnarray}
Consequently, noting $|\gamma| \leq [|\alpha|/2]$ and using
\eqref{C-1} in Lemma \ref{DecayEF} (slightly changing its proof by
using the cutoff function $\varphi^t$ to keep the wave operator
$\partial_t^2 - \Delta$), one can estimate the third line on the
right hand side of \eqref{F-1} as follows:
\begin{eqnarray}\label{F-2}
&&\sum_{\beta + \gamma = \alpha,
  |\gamma| \leq [|\alpha|/2]}\Pi_1(\varphi^t) \lesssim \langle
  t\rangle^{-\frac{3}{2}}\mathcal{E}_{
  [|\alpha|/2] + 4}^{\frac{1}{2}}\mathcal{E}_{
  |\alpha| + 2}^{\frac{1}{2}}\\\nonumber
&&\quad\quad\quad\quad\quad\quad +\ \langle t
  \rangle^{-\frac{3}{2}}\mathcal{E}_{
  |\alpha| + 1}^{\frac{1}{2}}\|\langle t\rangle (\partial_t^2 - \Delta)
  \Gamma^{\leq[|\alpha|/2] + 2} Y\|_{L^2}\\\nonumber
&&\quad\quad\quad\quad\quad\quad +\ \langle
t\rangle^{-\frac{3}{2}}\mathcal{E}_{
  [|\alpha|/2] + 3}^{\frac{1}{2}}
  \|\langle t\rangle(\partial_t^2
  - \Delta)\Gamma^{\leq|\alpha|} Y\|_{L^2}.
\end{eqnarray}
As we have already mentioned, $\Pi_2(\varphi^t)$ in the last line on the right hand side of (4.12) can be
bounded similarly as $\Pi_2(\varphi^t)$.

It remains to estimate the second line in \eqref{F-1}. Using Sobolev
imbedding $H^2 \hookrightarrow L^\infty$, we have
\begin{eqnarray}\label{F-9}
&&\sum_{\beta + \gamma = \alpha,\
  \gamma \neq \alpha}\big\||\nabla\Gamma^\beta Y||(\partial_t^2
   - \Delta)\Gamma^\gamma Y|\big\|_{L^2}\\\nonumber
&&\lesssim \sum_{\beta + \gamma = \alpha,\
  |\gamma| \leq [|\alpha|/2]}\|\nabla\Gamma^\beta Y\|_{L^2}\|(\partial_t^2
   - \Delta)\Gamma^\gamma Y\|_{L^\infty}\\\nonumber
&&\quad +\ \sum_{\beta + \gamma = \alpha,\
  [|\alpha|/2] < |\gamma| < |\alpha|}\|\nabla\Gamma^\beta Y\|_{L^\infty}\|(\partial_t^2
   - \Delta)\Gamma^\gamma Y\|_{L^2}\\\nonumber
&&\lesssim \langle t\rangle^{-1}\mathcal{E}_{
  |\alpha| + 1}^{\frac{1}{2}}\|\langle t\rangle(\partial_t^2
   - \Delta)\Gamma^{\leq[|\alpha|/2] + 2} Y\|_{L^2}\\\nonumber
&&\quad +\ \langle t\rangle^{-1}\mathcal{E}_{
  [|\alpha|/2] + 3}^{\frac{1}{2}}\|\langle t\rangle(\partial_t^2
   - \Delta)\Gamma^{\leq|\alpha| - 1} Y\|_{L^2}.
\end{eqnarray}
We now insert \eqref{F-9} and \eqref{F-2} into \eqref{F-1} to obtain
that
\begin{eqnarray}\label{F-4}
&&\Pi(|\alpha| + 2) \lesssim \langle
  t\rangle^{-1}\mathcal{E}_{ |\alpha| +
  2}^{\frac{1}{2}}\big(\mathcal{E}_{[|\alpha|/2]
  + 4}^{\frac{1}{2}} + \mathcal{X}_{[|\alpha|/2]
  + 4}^{\frac{1}{2}}\big)\\\nonumber
&&\quad\quad\quad\quad\ \ +\ \langle t\rangle^{-1}\mathcal{E}_{
  |\alpha| + 1}^{\frac{1}{2}}\|\langle t\rangle(\partial_t^2
   - \Delta)\Gamma^{\leq[|\alpha|/2] + 2} Y\|_{L^2}\\\nonumber
&&\quad\quad\quad\quad\ \ +\ \langle t\rangle^{-1}\mathcal{E}_{
  [|\alpha|/2] + 3}^{\frac{1}{2}}\|\langle t\rangle(\partial_t^2
   - \Delta)\Gamma^{\leq|\alpha|} Y\|_{L^2}
\end{eqnarray}

Now for $\kappa \geq 10$ and $|\alpha| \leq \kappa - 4$, one has
$|\alpha| + 2 \leq \kappa - 2$ and $[|\alpha|/2] + 4 \leq \kappa -
3$. Hence, by \eqref{F-4}, we have
\begin{eqnarray}\nonumber
&&\Pi(\kappa - 2) \lesssim \langle t\rangle^{-
  1}\mathcal{E}_{\kappa - 2}^{\frac{1}{2}}\big(\mathcal{E}_{ \kappa
  - 3}^{\frac{1}{2}} + \mathcal{X}_{
  \kappa - 3}^{\frac{1}{2}}\big)\\\nonumber
&&\quad\quad\quad \quad\ +\ \langle t\rangle^{-1}\mathcal{E}_{
  \kappa - 3}^{\frac{1}{2}}\|\langle t\rangle(\partial_t^2
   - \Delta)\Gamma^{\leq\kappa - 4} Y\|_{L^2}.
\end{eqnarray}
Inserting the above inequality into the estimate in Lemma \ref{SN-1}
and noting that $\mathcal{E}_{ \kappa - 2}(Y) \leq \delta$, one has
\begin{equation}\label{D-6}
\langle t\rangle\|\nabla\Gamma^{\leq\kappa - 4} p\|_{L^2} + \langle
t\rangle\|(\partial_t^2 - \Delta)\Gamma^{\leq\kappa - 4} Y\|_{L^2}
\lesssim \mathcal{E}_{ \kappa - 2}^{\frac{1}{2}}\big(\mathcal{E}_{
\kappa - 2}^{\frac{1}{2}} + \mathcal{X}_{\kappa -
2}^{\frac{1}{2}}\big).
\end{equation}
This proves the first estimate in Lemma \ref{SN-2}.

Next, for $|\alpha| \leq \kappa - 2$, there holds $[|\alpha|/2] + 4
\leq \kappa - 2$. Hence, one can derive from \eqref{F-4} that
\begin{eqnarray}\nonumber
&&\Pi(\kappa) \lesssim \langle t\rangle^{- 1}\mathcal{E}_{
  \kappa}^{\frac{1}{2}}\big(\mathcal{E}_{
  \kappa - 2}^{\frac{1}{2}} + \mathcal{X}_{
  \kappa - 2}^{\frac{1}{2}}\big)\\\nonumber
&&\quad\quad\  +\ \langle t\rangle^{-1}\mathcal{E}_{
  \kappa}^{\frac{1}{2}}\|\langle t\rangle(\partial_t^2
   - \Delta)\Gamma^{\leq\kappa - 4} Y\|_{L^2}\\\nonumber
&&\quad\quad\  +\ \langle t\rangle^{-1}\mathcal{E}_{
  \kappa - 2}^{\frac{1}{2}}\|\langle t\rangle(\partial_t^2
   - \Delta)\Gamma^{\leq\kappa - 2} Y\|_{L^2}.
\end{eqnarray}
which, combining \eqref{D-6} and Lemma \ref{SN-1} gives the second
estimate in the lemma.

\end{proof}

We are ready to state the following lemma:
\begin{lem}\label{WE-2}
Suppose that $\kappa \geq 10$. There exists $\delta > 0$ such that
if $\mathcal{E}_{\kappa - 2} \leq \delta$, then there hold
\begin{equation}\nonumber
\mathcal{X}_{\kappa - 2} \lesssim \mathcal{E}_{\kappa - 2},\quad
\mathcal{X}_{\kappa} \lesssim \mathcal{E}_{\kappa}.
\end{equation}
\end{lem}
\begin{proof}
Applying Lemma \ref{WE-1} and Lemma \ref{SN-2}, one has
\begin{eqnarray}\nonumber
\mathcal{X}_{\kappa - 2}^{\frac{1}{2}} &\lesssim& \mathcal{E}_{\kappa -
  2}^{\frac{1}{2}} + \langle t\rangle\|(\partial_t^2 -
  \Delta)\Gamma^{\leq\kappa - 4}Y\|_{L^2}\\\nonumber
&\lesssim& \mathcal{E}_{\kappa -
  2}^{\frac{1}{2}} + \mathcal{E}_{\kappa -
  2}^{\frac{1}{2}}\big(\mathcal{E}_{\kappa -
  2}^{\frac{1}{2}} + \mathcal{X}_{\kappa -
  2}^{\frac{1}{2}}\big),
\end{eqnarray}
which gives the first estimate of the lemma by noting the
assumption.

Next, applying Lemma \ref{WE-1} and Lemma \ref{SN-2} once more, one
has
\begin{eqnarray}\nonumber
\mathcal{X}_{\kappa}^{\frac{1}{2}} &\lesssim& \mathcal{E}_{
  \kappa}^{\frac{1}{2}} + \langle t\rangle\|(\partial_t^2 -
  \Delta)\Gamma^{\leq\kappa - 2}Y\|_{L^2}\\\nonumber
&\lesssim& \mathcal{E}_{\kappa}^{\frac{1}{2}} +
  \mathcal{E}_{\kappa}^{\frac{1}{2}}\big(\mathcal{E}_{\kappa -
  2}^{\frac{1}{2}} + \mathcal{X}_{\kappa -
  2}^{\frac{1}{2}}\big).
\end{eqnarray}
Then the second estimate of the lemma follows from the first one and
the assumption.
\end{proof}

\section{Higher-order Energy Estimate}\label{HEE}

This section is devoted to the higher-order generalized energy
estimate. We will see that the ghost weight method introduced by
Alinhac in \cite{Alinhac00} plays an important role.

Let $\kappa \geq 10$ and $|\alpha| \leq \kappa - 1$. Let $\sigma = t
- r$ and $q(\sigma) = \arctan\sigma$. Taking the $L^2$ inner product
of \eqref{Elasticity-D} with $e^{- q(\sigma)}\partial_t\Gamma^\alpha
Y$ and using integration by parts, we have
\begin{eqnarray}\nonumber
&&\frac{d}{dt}\int e^{- q(\sigma)}\big(|\partial_t
  \Gamma^\alpha Y|^2 + |\nabla \Gamma^\alpha Y|^2\big)dy\\\nonumber
&&= - \int \frac{e^{- q(\sigma)}}{1 + \sigma^2}\big(|\partial_t
  \Gamma^\alpha Y|^2 + |\nabla \Gamma^\alpha Y|^2\big)dy\\\nonumber
&&\quad +\ 2\int e^{- q(\sigma)}\partial_t
  \Gamma^\alpha Y\cdot(\partial_t^2 - \Delta)\Gamma^\alpha Ydy
  - 2\int\partial_j e^{- q(\sigma)}\partial_j \Gamma^\alpha
  Y\partial_t\Gamma^\alpha Y\big)dy\\\nonumber
&&= - \sum_j\int \frac{e^{- q(\sigma)}}{1 +
  \sigma^2}|(\omega_j\partial_t + \partial_j)
  \Gamma^\alpha Y|^2dy\\\nonumber
&&\quad -\ 2\int e^{- q(\sigma)}\partial_t
  \Gamma^\alpha Y\cdot\big[(\nabla X)^{- \top}\nabla\Gamma^\alpha p\big] dy\\\nonumber
&&\quad -\ 2\int e^{- q(\sigma)}\partial_t
  \Gamma^\alpha Y\cdot\sum_{\beta + \gamma
  = \alpha,\ \gamma \neq \alpha}C_\alpha^\beta(\nabla X)^{- \top}(\nabla\Gamma^\beta
  Y)^\top(\partial_t^2 - \Delta)\Gamma^\gamma Ydy,
\end{eqnarray}
which gives that
\begin{eqnarray}\label{ES-1}
&&\frac{d}{dt}\int e^{- q(\sigma)}\big(|\partial_t
  \Gamma^\alpha Y|^2 + |\nabla \Gamma^\alpha Y|^2\big)dy\\\nonumber
&&\quad +\ \sum_j\int \frac{e^{- q(\sigma)}}{1 +
  \sigma^2}|(\omega_j\partial_t + \partial_j)
  \Gamma^\alpha Y|^2dy\\\nonumber
&&\lesssim \mathcal{E}_{\kappa}^{\frac{1}{2}}
  \|(\nabla X)^{-T}\|_{L^\infty}\Big(\|\nabla\Gamma^\alpha p\|_{L^2} + \sum_{\beta + \gamma
  = \alpha,\ \gamma \neq \alpha}\|(\nabla
  \Gamma^\beta Y)^{\top}(\partial_t^2 - \Delta)\Gamma^\gamma
  Y\|_{L^2}\Big),
\end{eqnarray}
We will use the simple estimate: $\|(\nabla X)^{-T}\|_{L^\infty} \leq 4$. At the first glance, we will always lose one derivative since
$\|\nabla\Gamma^\alpha p\|_{L^2}$ contains the $|\alpha| + 2$
derivatives of $Y$. Fortunately, we may modify the proof for Lemma
\ref{SN-2} so that we have similar estimates but gain one derivative
and at the same time, lose $\langle t\rangle^{-\frac{1}{2}}$ decay
rate. Moreover, whenever we lose $\langle t\rangle^{-\frac{1}{2}}$
decay rate, we have a good derivative $\omega_j\partial_t +
\partial_j$. Then the ghost weight method of Alinhac enables us to take the advantage of
null structure of nonlinearities when we perform highest order
energy estimate. We emphasis that all of those calculations are
based on the physical structures of the system.

We still use the estimate in \eqref{F-1}, but we need refine the
last line of \eqref{F-5} as follows:
\begin{eqnarray}\label{F-6}
&&\nabla^\perp\cdot\big(\partial_j\Gamma^\gamma
  Y^2\phi^t(\omega_j\partial_t + \partial_j)\nabla
  \Gamma^\beta Y^1 + \partial_j\Gamma^\gamma Y^2\nabla\phi^t
  (\omega_j\partial_t + \partial_j)\Gamma^\beta Y^1\big)\\\nonumber
&&= \nabla^\perp\cdot\big(\partial_j\Gamma^\gamma
  Y^2\nabla[\phi^t(\omega_j\partial_t + \partial_j)
  \Gamma^\beta Y^1]\big)\\\nonumber
&&\quad +\ \nabla^\perp\cdot\big(\partial_j\Gamma^\gamma
  Y^2\phi^t\nabla\omega_j\partial_t\Gamma^\beta Y^1\big)\\\nonumber
&&= \nabla\cdot\big([\phi^t(\omega_j\partial_t + \partial_j)
  \Gamma^\beta Y^1]\nabla^\perp\partial_j\Gamma^\gamma
  Y^2\big)\\\nonumber
&&\quad +\ \nabla^\perp\cdot\big(\partial_j\Gamma^\gamma
  Y^2\phi^t\nabla\omega_j\partial_t\Gamma^\beta Y^1\big).
\end{eqnarray}
Replacing the last line  in \eqref{F-5} by \eqref{F-6}, one has
\begin{eqnarray}\label{F-7}
&&\nabla\cdot\big(\phi^t\partial_t\Gamma^\beta Y^1
  \nabla^\perp\partial_t\Gamma^\gamma Y^2 - \phi^t\partial_j\Gamma^\beta
  Y^1\nabla^\perp\partial_j\Gamma^\gamma Y^2\big)\\\nonumber
&&= \nabla\cdot\Big(\phi^t\partial_t\Gamma^\beta Y^1
  \nabla^\perp[\omega_j(\omega_j\partial_t + \partial_j)\Gamma^\gamma Y^2]\\\nonumber
&&\quad -\ \phi^t\partial_j\Gamma^\beta
  Y^1\nabla^\perp(\omega_j\partial_t + \partial_j)\Gamma^\gamma Y^2
  + \phi^t\partial_j
  \Gamma^\beta Y^1\nabla^\perp(\omega_j\partial_t +
  \partial_j)\Gamma^\gamma Y^2\\\nonumber
&&\quad -\ [\phi^t(\omega_j\partial_t + \partial_j)
  \Gamma^\beta Y^1]\nabla^\perp\partial_j\Gamma^\gamma
  Y^2\Big) - \nabla^\perp\cdot\big(\partial_j\Gamma^\gamma
  Y^2\phi^t\nabla\omega_j\partial_t\Gamma^\beta Y^1\big).
\end{eqnarray}
Using \eqref{F-7}, we may re-estimate $\Pi_1(\varphi^t)$ in
\eqref{F-1} as follows:
\begin{eqnarray}\nonumber
&&\big\|(- \Delta)^{-
  \frac{1}{2}}\nabla\cdot\big[\varphi^t\big(\partial_t
   \Gamma^\beta Y^1\nabla^\perp\partial_t\Gamma^\gamma
   Y^2 - \partial_j\Gamma^\beta
  Y^1\nabla^\perp\partial_j\Gamma^\gamma Y^2\big)
  \big]\big\|_{L^2}\\\nonumber
&&\lesssim \|D\Gamma^\beta Y\|_{L^2}\|1_{{\rm supp}\varphi^t}
  (\omega_j\partial_t + \partial_j)\nabla\Gamma^\gamma
  Y\|_{L^\infty}\\\nonumber
&&\quad +\ \langle t\rangle^{- \frac{1}{2}}\|1_{{\rm
  supp}\varphi^t}\langle r\rangle^{\frac{1}{2}}\langle t - r\rangle
  D^2\Gamma^\gamma Y\|_{L^\infty}\|\langle t - r\rangle^{-1}
  (\omega_j\partial_t + \partial_j)
  \Gamma^\beta Y\|_{L^2}\\\nonumber
&&\quad +\ \langle t\rangle^{-1}\|1_{{\rm
  supp}\varphi^t}D\Gamma^\gamma
  Y\|_{L^\infty}\|D\Gamma^\beta Y\|_{L^2}.
\end{eqnarray}
Again, the last term in the above inequality is due to the commutation of $\nabla$ with good derivatives. Using Lemma \ref{GoodDeri}  and Lemma \ref{DecayES}, the above is
bounded by
\begin{eqnarray}\nonumber
&&\langle t\rangle^{-1}\|D\Gamma^\beta Y\|_{L^2}\big(
  \|1_{{\rm supp}\varphi^t}\langle t\rangle (\partial_t^2 - \Delta)
  \Gamma^\gamma Y\|_{L^\infty} + \|1_{{\rm supp}\varphi^t}D\Gamma^{\leq
  |\gamma| + 1} Y\|_{L^\infty}\big)\\\nonumber
&&\quad +\ \langle t\rangle^{-\frac{1}{2}}
  \big(\mathcal{X}_{|\gamma| + 4} +
  \mathcal{E}_{|\gamma| + 3}\big)^{\frac{1}{2}}\|\langle t - r\rangle^{-1}
  (\omega_j\partial_t + \partial_j)
  \Gamma^\beta Y\|_{L^2}\\\nonumber
&&\quad +\ \langle t\rangle^{- 1}\|D\Gamma^\gamma
  Y\|_{L^\infty}\|D\Gamma^\beta Y\|_{L^2}.
\end{eqnarray}
Notice that $|\gamma| \leq [|\alpha|/2]$. We further use Lemma
\ref{DecayEF} (again, we need modify the proof slightly by adding a
cutoff function $\varphi^t$ to keep the wave operator $\partial_t^2
- \Delta$) and Lemma \ref{WE-2} to bound the above by
\begin{eqnarray}\nonumber
&&\langle t\rangle^{-\frac{1}{2}}
  \mathcal{E}_{|\gamma| + 4}^{\frac{1}{2}}\|\langle t - r\rangle^{-1}
  (\omega_j\partial_t + \partial_j)
  \Gamma^\beta Y\|_{L^2}\\\nonumber
&&\quad +\ \langle t\rangle^{- \frac{3}{2}}\mathcal{E}_{|\alpha|
  + 1}^{\frac{1}{2}}\|\langle t\rangle (\partial_t^2 - \Delta)
  \Gamma^{\leq[|\alpha|/2] + 2} Y\|_{L^2}+ \langle t\rangle^{- \frac{3}{2}}\mathcal{E}_{[|\alpha|/2] +
  4}^{\frac{1}{2}}\mathcal{E}_{|\alpha| + 1}^{\frac{1}{2}}.
\end{eqnarray}
Since $[|\alpha|/2] + 4 \leq \kappa - 2$, we can use Lemma
\ref{SN-2} and Lemma \ref{WE-2} to bound the above quantities by
\begin{eqnarray}\label{F-8}
\langle t\rangle^{-\frac{1}{2}}
  \mathcal{E}_{\kappa -
  2}^{\frac{1}{2}}\|\langle t - r\rangle^{-1}
  (\omega_j\partial_t + \partial_j)
  \Gamma^\beta Y\|_{L^2} + \langle t\rangle^{- \frac{3}{2}}\mathcal{E}_{\kappa -
  2}^{\frac{1}{2}}\mathcal{E}_{\kappa}^{\frac{1}{2}}.
\end{eqnarray}

Similarly, the last line in \eqref{F-1} can also be bounded by the
quantity in \eqref{F-8}. Consequently, we can derive by inserting
\eqref{F-8} into \eqref{F-1} that
\begin{eqnarray}\nonumber
&&\|\nabla\Gamma^\alpha p\|_{L^2} \lesssim \sum_{\beta + \gamma =
\alpha,\
  \gamma \neq \alpha}\big\||\nabla\Gamma^\beta Y||(\partial_t^2
   - \Delta)\Gamma^\gamma Y|\big\|_{L^2}\\\nonumber
&&\quad +\ \langle t\rangle^{-\frac{1}{2}}
  \mathcal{E}_{\kappa -
  2}^{\frac{1}{2}}\|\langle t - r\rangle^{-1}
  (\omega_j\partial_t + \partial_j)
  \Gamma^\beta Y\|_{L^2} + \langle t\rangle^{- 1}\mathcal{E}_{\kappa -
  2}^{\frac{1}{2}}\mathcal{E}_{\kappa}^{\frac{1}{2}}.
\end{eqnarray}
Inserting the above estimate into \eqref{ES-1}, we have
\begin{eqnarray}\label{ES-3}
&&\frac{d}{dt}\int e^{- q(\sigma)}\big(|\partial_t
  \Gamma^\alpha Y|^2 + |\nabla \Gamma^\alpha Y|^2\big)dy\\\nonumber
&&\quad +\ \int \frac{e^{- q(\sigma)}}{1 + \sigma^2}\big(|\omega\partial_t
  \Gamma^\alpha Y + \nabla \Gamma^\alpha Y|^2\big)dy\\\nonumber
&&\lesssim \mathcal{E}_{\kappa}^{\frac{1}{2}}\sum_{\beta + \gamma
  = \alpha,\ \gamma \neq \alpha}\|(\nabla
  \Gamma^\beta Y)^{\top}(\partial_t^2 - \Delta)\Gamma^\gamma
  Y\|_{L^2}\\\nonumber
&&\quad +\ \langle t\rangle^{-\frac{1}{2}}
  \mathcal{E}_{\kappa - 2}^{\frac{1}{2}}\mathcal{E}_{\kappa}^{\frac{1}{2}}
  \|\langle t - r\rangle^{-1}(\omega_j\partial_t + \partial_j)
  \Gamma^\beta Y\|_{L^2} + \langle t\rangle^{- 1}\mathcal{E}_{\kappa -
  2}^{\frac{1}{2}}\mathcal{E}_{\kappa}.
\end{eqnarray}

Now let us estimate the remaining terms in \eqref{ES-3}. Using Lemma
\ref{DecayEF} and Lemma \ref{SN-2}, it is easy to derive that
\begin{eqnarray}\nonumber
&&\sum_{\beta + \gamma
  = \alpha,\ \gamma \neq \alpha}\|(\nabla
  \Gamma^\beta Y)^{\top}(\partial_t^2 - \Delta)\Gamma^\gamma
  Y\|_{L^2}\\\nonumber
&&\lesssim \sum_{\beta + \gamma
  = \alpha,\ |\gamma| \leq [|\alpha|/2]}\|\nabla
  \Gamma^\beta Y\|_{L^2}\|(\partial_t^2 - \Delta)\Gamma^\gamma
  Y\|_{L^\infty}\\\nonumber
&&\quad +\ \sum_{\beta + \gamma
  = \alpha,\ |\beta| \leq [|\alpha|/2], \gamma \neq \alpha}\|\nabla
  \Gamma^\beta Y\|_{L^\infty}\|(\partial_t^2 - \Delta)\Gamma^\gamma
  Y\|_{L^2}\\\nonumber
&&\lesssim \mathcal{E}_{\kappa}^{\frac{1}{2}}\|(\partial_t^2 -
  \Delta)\Gamma^{\leq\kappa - 4}Y\|_{L^2} + \langle t\rangle^{-
  \frac{1}{2}}\mathcal{E}_{\kappa - 2}^{\frac{1}{2}}\|(\partial_t^2 -
  \Delta)\Gamma^{\leq\kappa - 2}Y\|_{L^2}\\\nonumber
&&\lesssim \langle t\rangle^{-
  1}\mathcal{E}_{\kappa}^{\frac{1}{2}}\mathcal{E}_{\kappa - 2}.
\end{eqnarray}
Inserting the above estimates into \eqref{ES-3} and using Cauchy
inequality, we have
\begin{eqnarray}\label{ES-2}
\frac{d}{dt}\sum_{|\alpha| \leq \kappa - 1}\int e^{-
q(\sigma)}\big(|\partial_t \Gamma^\alpha Y|^2 + |\nabla
\Gamma^\alpha Y|^2\big)dy \lesssim \langle
t\rangle^{-1}\mathcal{E}_{\kappa} \mathcal{E}_{\kappa -
2}^{\frac{1}{2}}.
\end{eqnarray}
Here we used $\mathcal{E}_{\kappa -
2} \leq 1$. This gives the first differential inequality \eqref{C1} at the end of section
\ref{Equations}.

\section{Lower-order Energy Estimate}\label{LEE}

In this section we perform the lower order energy estimate. Let
$|\alpha| \leq \kappa - 3$. We rewrite \eqref{Elasticity-D} as
\begin{eqnarray}\nonumber
&&(\nabla X)^{\top}(\partial_t^2 - \Delta) \Gamma^\alpha
  Y + \nabla\Gamma^\alpha p\\\nonumber
&&= - \sum_{\beta + \gamma = \alpha,\ \gamma \neq \alpha}C_\alpha^\beta
  (\nabla \Gamma^\beta Y)^{\top}(\partial_t^2 -
   \Delta)\Gamma^\gamma Y.
\end{eqnarray}
Applying the curl operator to the above equation, one has
\begin{eqnarray}\nonumber
(\partial_t^2 - \Delta)\nabla^\perp\cdot
  \Gamma^\alpha Y &=& - \sum_{\beta + \gamma = \alpha,\ \gamma \neq \alpha}C_\alpha^\beta\big\{
  \nabla^\perp\cdot\big[(\nabla \Gamma^\beta
  Y)^{\top}(\partial_t^2 - \Delta)\Gamma^\gamma Y\big]\\\nonumber
&&-\ \nabla^\perp\cdot\big[(\nabla
  Y)^{\top}(\partial_t^2 - \Delta)\Gamma^\alpha Y\big]\big\}.
\end{eqnarray}
Consequently, we have
\begin{eqnarray}\label{Elasticity-L}
&&(\partial_t^2 - \Delta)(- \Delta)^{- \frac{1}{2}}\nabla^\perp\cdot
  \Gamma^\alpha Y\\\nonumber
&&=  \sum_{\beta + \gamma = \alpha}C_\alpha^\beta
  (- \Delta)^{- \frac{1}{2}}\nabla^\perp\cdot\big[(\nabla \Gamma^\beta
  Y)^{\top}(\partial_t^2 - \Delta)\Gamma^\gamma Y\big].
\end{eqnarray}

Multiplying \eqref{Elasticity-L} by $\partial_t(- \Delta)^{-
\frac{1}{2}}\nabla^\perp\cdot
  \Gamma^\alpha Y$ and then integrating over $\mathbb{R}^2$, one has
\begin{eqnarray}\nonumber
&&\frac{1}{2}\frac{d}{dt}\int\big(\big|\partial_t(- \Delta)^{-
  \frac{1}{2}}\nabla^\perp\cdot\Gamma^\alpha Y\big|^2
  + \big|\nabla(- \Delta)^{-\frac{1}{2}}\nabla^\perp\cdot\Gamma^\alpha
  Y\big|^2\big)dy\\\nonumber
&&\lesssim \sum_{\beta + \gamma = \alpha}
  \big\|\partial_t(- \Delta)^{- \frac{1}{2}}\nabla^\perp\cdot
  \Gamma^\alpha Y\big\|_{L^2}\big\||\nabla \Gamma^\beta
  Y||(\partial_t^2 - \Delta)\Gamma^\gamma Y|\big\|_{L^2}\\\nonumber
&&\lesssim \mathcal{E}_{\kappa - 2}^{\frac{1}{2}}\sum_{\beta +
  \gamma = \alpha}\big\||\nabla \Gamma^\beta
  Y||(\partial_t^2 - \Delta)\Gamma^\gamma Y|\big\|_{L^2}.
\end{eqnarray}
Let us first use Lemma \ref{DecayEF}, Lemma \ref{SN-2} and Lemma
\ref{WE-2} to estimate that
\begin{eqnarray}\nonumber
&&\sum_{\beta + \gamma = \alpha,\ |\beta| \leq [|\alpha|/2]
  }\big\||\nabla\Gamma^\beta Y||(\partial_t^2 - \Delta)
  \Gamma^\gamma Y|\big\|_{L^2}\\\nonumber
&&\quad\quad \quad \lesssim \|\nabla\Gamma^{\leq[|\alpha|/2]}
  Y\|_{L^\infty}\|(\partial_t^2 - \Delta)
  \Gamma^{\leq|\alpha|} Y\|_{L^2}\\\nonumber
&&\quad\quad \quad \lesssim \langle t\rangle^{-
3/2}\mathcal{E}_{\kappa -
  2}^{\frac{1}{2}}\|\langle t\rangle(\partial_t^2 - \Delta)
  \Gamma^{\leq\kappa - 2} Y\|_{L^2}\\\nonumber
&&\quad\quad \quad \lesssim \langle t\rangle^{-
3/2}\mathcal{E}_{\kappa -
  2}\mathcal{E}_{\kappa}^{\frac{1}{2}}.
\end{eqnarray}
Similarly, by $[|\alpha|/2] + 4 \leq \kappa - 2$, one also has
\begin{eqnarray}\nonumber
&&\sum_{\beta + \gamma = \alpha,\ |\gamma| \leq
  [|\alpha|/2]}\big\||\nabla
  \Gamma^\beta Y||(\partial_t^2 - \Delta)
  \Gamma^\gamma Y|\big\|_{L^2}\\\nonumber
&&\quad\quad \quad \lesssim \|(1 - \varphi^t)\nabla \Gamma^{\leq\kappa -
  3}Y\|_{L^\infty}\|(1 - \varphi^t)(\partial_t^2 - \Delta)
  \Gamma^{\leq[|\alpha|/2]} Y\|_{L^2}\\\nonumber
&&\quad\quad \quad \quad +\ \|\varphi^t\nabla \Gamma^{\leq\kappa - 3}
  Y\|_{L^2}\|\varphi^t(\partial_t^2 - \Delta)
  \Gamma^{\leq[|\alpha|/2]} Y\|_{L^\infty}\\\nonumber
&&\quad\quad \quad \lesssim \langle t\rangle^{- 3/2}
  \mathcal{E}_{\kappa}^{\frac{1}{2}}
  \|\langle t\rangle(1 - \varphi^t)(\partial_t^2 -
  \Delta)\Gamma^{\leq[|\alpha|/2]} Y\|_{L^2}\\\nonumber
&&\quad\quad \quad \quad +\ \langle t\rangle^{- 3/2}
  \mathcal{E}_{\kappa - 2}^{\frac{1}{2}}\|\langle t\rangle\varphi^t(\partial_t^2 - \Delta)
  \Gamma^{\leq[|\alpha|/2] + 2} Y\|_{L^2}\\\nonumber
&&\quad\quad \quad \lesssim \langle t\rangle^{- 3/2}
  \mathcal{E}_{\kappa - 2}\mathcal{E}_{\kappa}^{\frac{1}{2}}.
\end{eqnarray}
Hence, we have
\begin{eqnarray}\label{LOEE-1}
\frac{d}{dt}\sum_{|\alpha| \leq \kappa - 3}\int\big|D(- \Delta)^{-
\frac{1}{2}}\nabla^\perp\cdot\Gamma^\alpha Y\big|^2dy \lesssim
\langle t\rangle^{- 3/2}\mathcal{E}_{\kappa}^{\frac{1}{2}}
\mathcal{E}_{\kappa - 2}.
\end{eqnarray}

Now let us estimate $(- \Delta)^{- \frac{1}{2}}\nabla\cdot
D\Gamma^\alpha Y$. Using \eqref{Struc-2}, one has
\begin{eqnarray}\nonumber
&&(- \Delta)^{- \frac{1}{2}}\nabla\cdot D\Gamma^\alpha Y\\\nonumber
&&= (- \Delta)^{- \frac{1}{2}}D\sum_{\beta + \gamma = \alpha}C_\alpha^\beta
  \big[\partial_1\Gamma^\beta  Y^2\partial_2\Gamma^\gamma  Y^1
  - \partial_1\Gamma^\gamma  Y^1\partial_2\Gamma^\beta  Y^2\big]\\\nonumber
&&= (- \Delta)^{- \frac{1}{2}}\sum_{\beta + \gamma = \alpha}C_\alpha^\beta
  \big[\partial_1D\Gamma^\beta  Y^2\partial_2\Gamma^\gamma  Y^1
  - \partial_1\Gamma^\gamma  Y^1\partial_2D\Gamma^\beta  Y^2\big]\\\nonumber
&&\quad +\ (- \Delta)^{- \frac{1}{2}}\sum_{\beta + \gamma = \alpha}C_\alpha^\beta
  \big[\partial_1\Gamma^\beta  Y^2\partial_2D\Gamma^\gamma  Y^1
  - \partial_1D\Gamma^\gamma  Y^1\partial_2\Gamma^\beta  Y^2\big]\\\nonumber
&&= (- \Delta)^{- \frac{1}{2}}\nabla^\perp\cdot\sum_{\beta + \gamma
  = \alpha}C_\alpha^\beta\big[D\Gamma^\beta  Y^2\nabla\Gamma^\gamma  Y^1
  - \nabla\Gamma^\beta  Y^2D\Gamma^\gamma  Y^1\big].
\end{eqnarray}
Hence, one has
\begin{eqnarray}\label{B-3}
\|(- \Delta)^{- \frac{1}{2}}\nabla\cdot D\Gamma^\alpha
  Y\|_{L^2} &\lesssim& \sum_{\beta + \gamma = \alpha}\big\||D\Gamma^\beta
  Y||D\Gamma^\gamma Y|\big\|_{L^2}\\\nonumber
&\lesssim& \|D\Gamma^{\leq|\alpha|}
  Y\|_{L^2}\|D\Gamma^{\leq[|\alpha|/2]}\|_{L^\infty}\\\nonumber
&\lesssim& \mathcal{E}_{\kappa - 2}.
\end{eqnarray}
Hence, we see that $\big\|(- \Delta)^{-
\frac{1}{2}}\nabla^\perp\cdot D\Gamma^{\leq\kappa - 3}Y\big\|_{L^2}$ is
equivalent of $\big\|D\Gamma^{\leq\kappa - 3}Y\big\|_{L^2}$ since
\begin{eqnarray}\nonumber
\Big|\big\|D\Gamma^{\leq\kappa - 3}Y\big\|_{L^2}^2 - \big\|(- \Delta)^{-
\frac{1}{2}}\nabla^\perp\cdot D\Gamma^{\leq\kappa -
3}Y\big\|_{L^2}^2\Big| \lesssim \mathcal{E}_{\kappa - 3}^2 \lesssim
\epsilon^2\mathcal{E}_{\kappa - 2}.
\end{eqnarray}
Then we can replace all $\mathcal{E}_{\kappa - 2}$ appeared
throughout this paper by $\big\|(- \Delta)^{-
\frac{1}{2}}\nabla^\perp\cdot D\Gamma^{\leq\kappa - 3}Y\big\|_{L^2}^2$
without changing the final result. Then \eqref{LOEE-1} gives the
second differential inequality \eqref{C2} at the end of section
\ref{Equations}.

\section*{Appendix A}

In this appendix we explain how to obtain \eqref{Elasticity-D} and \eqref{Struc-2}. Let
$$\Gamma_i \in \{\partial_t, \partial_1, \partial_2, \widetilde{\Omega}, \widetilde{S}\},\ i = 1, 2, \cdots, 5.$$
Recall that we have defined the scaling and rotation groups in section 2 so that their generators are $\widetilde{S}$ and $\widetilde{\Omega}$. Similarly, we can define translation groups so that their generators are $\partial_t$, $\partial_1$ and $\partial_2$.

Now for each multi-index $\alpha$ and
$$\Gamma^\alpha = \Gamma_1^{\alpha_1}\cdots \widetilde{S}^{\alpha_j}\cdots\Gamma_5^{\alpha_5},$$ we can naturally define the group $T_{\alpha}$ such that
\begin{equation}\label{F1}
\frac{d^{\alpha_1}}{d\lambda_1^{\alpha_1}}\cdots\frac{d^{\alpha_5}}{d\lambda_5^{\alpha_5}} T_\alpha X\big|_{(\lambda_1,\cdots, \lambda_5) = e_j} = \Gamma^\alpha X.\end{equation}
Here $e_j$ is the unit vector in $\mathbb{R}^5$ whose $j$-th component 1. Indeed, $T_\alpha X$ can be defined as follows:
$$T_\alpha X = (T_{\lambda_1})^{\alpha_1}\cdots (T_{\lambda_5})^{\alpha_5}X,$$
where each group $T_{\lambda_j}$ has a generator $\Gamma_j$. Similar definition is applied to $T_\alpha p$.

Due to the invariance property of the system, one has the fact that $(T_\alpha X, T_\alpha p)$ is still a solution of the system \eqref{B-9}. Consequently, we have
\begin{equation}\nonumber
\begin{cases}
(\nabla T_\alpha X)^{\top}(\partial_t^2T_\alpha Y - \Delta T_\alpha Y) = -\nabla T_\alpha p,\\[-4mm]\\
\nabla\cdot T_\alpha Y = - \det(\nabla T_\alpha Y).
\end{cases}
\end{equation}
Clearly, differentiating the above equations with respect to $\lambda_j$'s and then use \eqref{F1}, one deduces \eqref{Elasticity-D} and \eqref{Struc-2}.

\section*{Acknowledgement.}

The author would like to thank Prof. Jean Bourgain, Prof. Thomas
Sideris and Prof. Yi Zhou for many stimulating discussions, and thank the anonymous referees for constructive suggestions and Dr. Qingtian Zhang from Penn State for checking the English writing. He also
would like to thank Prof. Weinan E, Prof. Jiaxing Hong and Prof.
Fanghua Lin for their long-standing supports  and encouragements
during the past several years when the author was working on this
problem.  Part of this work was carried out when the author was a
member of IAS during the spring of 2014. He would like to thank the
hospitality of the Institute. His work was partially supported by
the S. S. Chern Foundation for Mathematics Research Fund and the
Charles Simonyi Endowment via the IAS. The author was also in part
supported by NSFC (grant No. 11421061 and 11222107), National Support Program for Young Top-Notch Talents, Shanghai Shu Guang project,
Shanghai Talent Development Fund and
SGST 09DZ2272900.


\end{document}